\documentclass[11pt,centertags,reqno]{amsart}

\usepackage{amssymb}
\usepackage[english]{babel}
\usepackage[T1]{fontenc}
\usepackage{amsthm}
\usepackage{verbatim}
\usepackage{fancyhdr}
\usepackage[utf8]{inputenc} 
\usepackage[numbers]{natbib}
\usepackage{url}

\def \F {{\mathcal F}}

\def \H {{\mathcal H}}

\def \N {{\mathbb N}}
\def \P {{\mathbb P}}

\def \R {{\mathbb R}}

\def \I {{\mathbf 1}}
\def \bF {{\mathbb F}}

\def \bH {{\mathbb H}}

\newtheorem{theorem}{Theorem}[section]
\newtheorem{lemma}[theorem]{Lemma}
\newtheorem{definition}[theorem]{Definition}
\newtheorem{remark}[theorem]{Remark}
\newtheorem{proposition}[theorem]{Proposition}
\newtheorem{example}[theorem]{Example}
\newtheorem{ass}[theorem]{Assumption}

\newcommand{\ud}{\mathrm d}

\newcommand{\esp}[2][\mathbb E] {#1\left[#2\right]}

\newcommand{\condespf}[2][\F_t]       {\mathbb E\left.\left[#2\right|#1\right]}

\newcommand{\condesph}[2][\H_t]       {\mathbb E\left.\left[#2\right|#1\right]}
\newcommand{\condesphh}[2][\H_{\tau-}]       {\mathbb E\left.\left[#2\right|#1\right]}

\newcommand{\condesphti}[2][\H_{t_i}]       {\mathbb E\left.\left[#2\right|#1\right]}

\newcommand{\varfti}[2][\F_{t_i}]       {\mathrm Var\left.\left[#2\right|#1\right]}
\newcommand{\covfti}[2][\F_{t_i}]       {\mathrm Cov\left.\left[#2\right|#1\right]}
\newcommand{\condespfti}[2][\F_{t_i}]       {\mathbb E\left.\left[#2\right|#1\right]}

\numberwithin{equation}{section}\makeatletter
\renewcommand{\subsection}{\@startsection
{subsection}{2}{0mm}{\baselineskip}{-0.25cm}
{\normalfont\normalsize\bf}} \makeatother

\author[C.~Ceci]{Claudia  Ceci}
\author[A.~Cretarola]{Alessandra Cretarola}
\author[F.~Russo]{Francesco Russo}

\begin{document}

\address{Claudia Ceci, Dipartimento di Economia,
Università degli Studi ``G. D'Annunzio'' di Chieti-Pescara, Viale Pindaro 42,
I-65127 Pescara, Italy.}
\email{c.ceci@unich.it}

\address{Alessandra Cretarola, Dipartimento di Matematica e Informatica, Università degli Studi di Perugia, via Vanvitelli, 1, I-06123 Perugia, Italy.} \email{alessandra.cretarola@dmi.unipg.it}

\address{Francesco Russo, Ecole Nationale 
Supérieure
 des Techniques Avancées,
ENSTA-ParisTech
Unité de Mathématiques appliquées, 828, boulevard des Maréchaux,
F-91120 Palaiseau
} 
\email{francesco.russo@ensta-paristech.fr}

\title[ BSDEs under partial information and financial applications]{BSDEs under partial information and financial applications}

\date{}

\begin{abstract}
\begin{center}
In this paper we provide existence and uniqueness results for the solution of BSDEs driven by a general square integrable
 martingale under partial information. We discuss some special cases where the solution to a BSDE under restricted information can be derived by that related to a  problem of a BSDE under full information. In particular, we provide  a suitable version of the F\"ollmer-Schweizer decomposition of a square integrable random variable working under partial information and we use this achievement to investigate the local risk-minimization approach for a semimartingale financial market model. 
 \end{center}
\end{abstract}



\maketitle


\section{Introduction}

\noindent 
The goal of this paper is  to provide existence and uniqueness results for backward stochastic differential equations (in short BSDEs) driven by a general càdlàg square integrable martingale under partial information and to apply such results to provide a financial application. 

Frameworks affected by incomplete information
represent an interesting issue arising in many problems. Mathematically, this  means to consider an additional filtration $\bH=(\H_t)_{0 \leq t \leq T}$ smaller than the full information flow $\bF=(\F_t)_{0 \leq t \leq T}$, with $T$ denoting a finite time horizon. A typical example arises when ${\mathcal H}_t = {\mathcal F}_{(t-\tau)^+}$
where $\tau \in (0,T)$ is a fixed delay and $(t-\tau)^+:=\max\{0, t-\tau\}$ with $t \in [0,T]$, or in a financial market where the stock prices can only be observed at discrete time instants or their dynamics depends on an unobservable stochastic factor  and $\bH$ denotes the information available to investors (see for instance~\cite{c06},~\cite{ce12},~\cite{ce13},~\cite{f2000},~\cite{fs2012}).

For BSDEs driven by a general càdlàg martingale beyond the Brownian setting, there exist very few results in the literature (besides the pioneering work of~\cite{b93}, see~\cite{ms94},~\cite{ekh97} and more recently~\cite{ccr},~\cite{bdm02} and~\cite{cfs08}, as far as we are aware). In~\cite{ccr} the authors study for the first time such a general case when there are restrictions on the available information by focusing on BSDEs whose driver is equal to zero. Let $T \in (0,\infty)$ be  a fixed time horizon and $\xi$ a square-integrable  $\mathcal F_T$-measurable random variable which denotes the terminal condition. In this paper we consider general BSDEs of the form:
\begin{equation}\label{eq:bsde0}
Y_t=\xi + \int_t^T f(s,Y_{s-},Z_s)\ud \langle M\rangle_s - \int_t^T Z_s \ud M_s -(O_T - O_t), \quad 0 \leq t \leq T,
\end{equation}
driven by a square-integrable càdlàg $\bF$-martingale $M=(M_t)_{0 \leq t \leq T}$,  with $\bF$-predictable quadratic variation $ \langle M\rangle=(\langle M\rangle_t)_{0 \leq t \leq T}$,   
where $O= (O_t)_{0 \leq t \leq T}$  is a square-integrable $\bF$-martingale, satisfying a suitable orthogonality condition that we will make more precise in the next section. The driver of the equation is denoted by $f$ and for each $(y,z) \in \R \times \R$, the process $f(\cdot, \cdot, y, z)=(f(\cdot, t, y, z))_{0\leq t \leq T}$ is  $\bF$-predictable. 

We look for a solution $(Y,Z)$ to equation (\ref{eq:bsde0}) under partial information, where $Y = (Y_t)_{0 \leq t \leq T}$
is a  càdlàg $\bF$-adapted process such that  $\esp{\sup_{0 \leq t \leq T} |Y_t|^2}<\infty$ and $Z = (Z_t)_{0 \leq t \leq T}$ is an $\bH$-predictable process such that $\esp{\int_0^T|Z_s|^2\ud \langle M\rangle_s} < \infty$.  
 
Our first important achievement, stated in Theorem \ref{th:ex-uniq}, concerns existence and uniqueness properties of the solution to such BSDEs. We get such results by assuming $f$ uniformly Lipschitz with respect to $(y,z)$ and the behavior of
$\langle M\rangle$ to be controlled by a deterministic function. Moreover, we provide  in Proposition \ref{nuova} a  representation of the solution to BSDEs under restricted information in terms of the Radon-Nikod\'ym derivative of two $\bH$-predictable dual projections involving the solution of a problem under full information. Thanks to this result, in the particular case where the driver $f$ does not depend on $z$, we give in Proposition \ref{nuova1} an explicit characterization of the solution to BSDEs under restricted information in terms of the solution to the corresponding BSDEs under full information. Finally, as an illustrative example, we discuss the special case of delayed information, that is, when ${\mathcal H}_t = {\mathcal F}_{(t-\tau)^+}$ for each $t \in[0,T]$, with $\tau \in (0,T)$ being fixed, 
once we assume that $\langle M\rangle$ and $f(\cdot, \cdot, y,z)$  are $\bH$-predictable processes and $f$ does not depend on $y$. Proposition \ref{2nuova} ensures  existence of the solution to the BSDEs under restricted information  by a constructive procedure under weaker conditions on $f$ with respect to the general theorem. 

As a financial application we discuss the local risk-minimization approach for partially observable semimartingale models.
The local risk-minimization approach is a quadratic hedging method for contingent claims in incomplete markets which keeps the replication constraint and looks for a hedging strategy (in general not self-financing) with minimal cost, see e.g.~\cite{fs} and~\cite{s01} for a further discussion on this issue. The study of this approach under partial information in full generality is still an interesting topic to discuss.   
The first step was done by~\cite{fs}, where they complete the information starting from the reference filtration and recover the optimal strategy by means of predictable projections with respect to the enlarged filtration.
Some further contributions in this direction can be found in~\cite{s94} and~\cite{ccr} in the case where the underlying price process is a (local) martingale under the real-world probability measure. 
In~\cite{s94}, the author provides an explicit expression for risk-minimizing hedging strategies under restricted information in terms of predictable dual projections, whereas in~\cite{ccr}, by proving a version of the Galtchouk-Kunita-Watanabe decomposition that works under partial information, the authors extend the results of~\cite{fs86} to the partial information framework and show how their result fits in the approach of~\cite{s94}. Furthermore, an application of the local risk-minimization approach in the case of incomplete information to defaultable markets in the sense of~\cite{fs} can be found in~\cite{bc09}.

Here, we consider a more general situation since we allow the underlying price process to be represented by a semimartingale under the real-world probability measure. More precisely, in Proposition  \ref{Claudia} we provide a version of the F\"ollmer-Schweizer decomposition of a square-integrable random variable (that typically represents the payoff of a contract) with respect to the underlying price process, that works under partial information. 

Then, we study the relationship between the F\"ollmer-Schweizer decomposition of a contingent claim under partial information and the
existence of a locally risk-minimizing strategy according to the partial information framework. 

In addition, we discuss the case where the underlying price process can exhibit jumps in the classical full information setting. 


The paper is organized as follows. In Section \ref{setting} we formulate the problem for BSDEs under partial information, we prove existence and uniqueness properties of solutions and we give the representation results in terms of $\bH$-predictable dual projections.  Section \ref{setting} concludes with a discussion of some special cases. Section \ref{LRM} is devoted to the study of local risk-minimization under partial information via BSDEs.
A discussion about the case of complete information in presence of jumps in the underlying price process can be found in Section \ref{sec:complete}. 
Finally, some detailed definitions and technical results are gathered in Section \ref{sec:tech} in Appendix.

\section{Backward stochastic differential equations under partial information} \label{setting}

\noindent Let us fix a probability space $(\Omega,\F,\P)$ endowed with a filtration $\bF:= (\F_t)_{0 \leq t \leq T}$, where $\F_t$
represents the full information at time $t$ and $T$ denotes a fixed and finite time horizon. We assume that $\F=\F_T$. Then we consider a subfiltration $\bH := (\H_t)_{0 \leq t \leq T}$
of $\bF$, i.e. $\H_t \subseteq \F_t$, for each $t \in [0,T]$, corresponding to the available information level. We
remark that both filtrations are assumed to satisfy the usual hypotheses of completeness and right-continuity, see e.g.~\cite{pp}.\\
For simplicity we only consider the one-dimensional case. Extensions to several dimensions
are straightforward and left to the reader. The data of the problem are:
\begin{itemize}
\item an $\R$-valued square-integrable (càdlàg) $\bF$-martingale $M=(M_t)_{0 \leq t \leq T}$ with $\bF$-predictable
quadratic variation process denoted by $\langle M\rangle  = (\langle M,M\rangle)_{0 \leq t \leq T}$;
\item a terminal condition $\xi \in L^2(\Omega,\F_T,\P;\R)$\footnote{The space $L^2(\Omega,\F_T,\P;\R)$ denotes the set of all real-valued $\F_T$-measurable  random variables $H$ such that $\esp{|H|^2} = \int_\Omega |H|^2\ud \P < \infty$.};
\item a coefficient $f:\Omega \times [0,T] \times \R \times \R \longrightarrow \R$, such that, for each $(y,z) \in \R \times \R$, the process $f(\cdot, \cdot, y, z)=(f(\cdot, t, y, z))_{0\leq t \leq T}$
is 
$\bF$-predictable. The random function $f$ is said to be the {\em driver} of the equation.
\end{itemize}
We make the following assumptions on the coefficient $f$.

\newpage 
\begin{ass} \label{ass:driver}
\begin{enumerate}
\item[]
\item[(i)] $f$ is uniformly Lipschitz with respect to $(y,z)$: there exists a constant $K \ge 0$ such that for every $(y,z),(y',z') \in \R \times \R$,
$$
|f(\omega,t,y,z)-f(\omega,t,y',z')| \leq K \left(|y-y'|+|z-z'|\right), \quad (\P \otimes \langle M\rangle)-\mbox{a.e.\ on}\ \Omega \times [0,T];
$$
\item[(ii)] the following integrability condition is satisfied:
$$
\esp{\int_0^T|f(t,0,0)|^2\ud \langle M\rangle_t} < \infty.
$$
\end{enumerate}
\end{ass}

\noindent To describe the parameters and the solution of BSDEs, we introduce the following spaces:
\begin{itemize}
\item $\mathcal S_\F^2(0,T)$, the set of all càdlàg $\bF$-adapted processes $\phi=(\phi_t)_{0\leq t \leq T}$ such that $\|\phi\|_{\mathcal S^2}^2:=\esp{\sup_{0 \leq t \leq T} |\phi_t|^2}<\infty$;
\item $\mathcal M_{\H}^2(0,T)$ ($\mathcal M_{\F}^2(0,T))$, the set of all $\bH$-predictable (respectively  $\bF$-predictable) processes $\varphi=(\varphi_t)_{0\leq t \leq T}$ such that $\|\varphi\|_{\mathcal M^2}^2:=\esp{\int_0^T|\varphi_s|^2\ud \langle M\rangle_s} < \infty$;
\item $\mathcal L_\F^2(0,T)$, the set of all 
$\bF$-martingales $\psi=(\psi_t)_{0\leq t \leq T}$ with $\psi_0 = 0$, such that $\|\psi\|_{\mathcal L^2}^2:=\esp{\langle \psi \rangle_T}=\esp{\psi_T^2}<\infty$.
\end{itemize}

We now give the definitions of solution in a full and in a partial information framework, respectively.

\begin{definition} \label{def1:solBSDE}
A solution of the BSDE
\begin{equation} \label{eq:bsdeF}
\tilde Y_t = \xi + \int_t^Tf(s,\tilde Y_{s-}, \tilde Z_s)\ud \langle M\rangle_s - \int_t^T \tilde Z_s \ud M_s - (\tilde O_T - \tilde O_t), \quad 0 \leq t \leq T,
\end{equation}
with data $(\xi,f)$ under complete information, is a triplet $(\tilde Y,\tilde Z, \tilde O)=(\tilde Y_t,\tilde Z_t, \tilde O_t)_{0 \leq t \leq T}$ of processes with values in $\R \times \R \times \R$ satisfying \eqref{eq:bsdeF}, such that
$$
(\tilde Y,\tilde Z, \tilde O) \in \mathcal S_\F^2(0,T) \times \mathcal M_\F^2(0,T) \times \mathcal L_\F^2(0,T),
$$
where $\tilde O$ is strongly orthogonal to $M$ (
i.e. $\langle \tilde O, M\rangle_t=0$ $\P$-a.s., for every $t \in [0,T]$).
\end{definition}

\begin{definition} \label{def:solBSDE}
A solution of the BSDE 
\begin{equation}\label{eq:bsde1}
Y_t=\xi + \int_t^T f(s,Y_{s-},Z_s)\ud \langle M\rangle_s - \int_t^T Z_s \ud M_s -(O_T - O_t), \quad 0 \leq t \leq T,
\end{equation}
with data $(\xi,f,\bH)$ under partial information, is a triplet $(Y,Z,O)=(Y_t,Z_t,O_t)_{0 \leq t \leq T}$ of processes with values in $\R \times \R \times \R$ satisfying \eqref{eq:bsde1}, such that
$$
(Y,Z,O) \in \mathcal S_\F^2(0,T) \times \mathcal M_{\H}^2(0,T) \times \mathcal L_\F^2(0,T),
$$
where $O$ satisfies
the orthogonality condition
\begin{equation} \label{eq:orthogcond}
\esp{O_T \int_0^T \varphi_t \ud M_t}=0,
\end{equation}
for all processes $\varphi \in \mathcal M_{\H}^2(0,T)$.
\end{definition}

\begin{remark}
Sometimes in the literature, only the couple $(Y,Z)$ identifies the solution of a BSDE of the form \eqref{eq:bsdeF} or  \eqref{eq:bsde1}. Indeed, this is reasonable since the $\bF$-martingale $O$ is uniquely determined by the processes $Y$ and $Z$ that satisfy the equation.
\end{remark}

\begin{remark} \label{R1}
The orthogonality condition given in \eqref{eq:orthogcond} is weaker than the classical strong orthogonality condition considered in Definition \ref{def1:solBSDE}.
 Indeed, set $N_t=\int_0^t \varphi_s\ud M_s$, for each $t \in [0,T]$, where $\varphi \in \mathcal M_{\H}^2(0,T)$. 
If $\psi \in \mathcal L_\F^2(0,T)$ is such that 
$$
\langle \psi,M\rangle_t=0\quad \P-\mbox{a.s.}, \quad \forall t \in [0,T],
$$
then
$$
\langle \psi,N\rangle_t=\int_0^t \varphi_s\ud \langle \psi,M\rangle_s=0\quad \P-\mbox{a.s.}, \quad \forall t \in [0,T].
$$  
Consequently, $\psi N$ is an $\bF$-martingale null at zero, that implies 
$$
\esp{\psi_t N_t}=0, \quad \forall t \in [0,T],
$$ 
and in particular condition \eqref{eq:orthogcond}.
\end{remark}

\begin{remark} \label{rem:orth}
Let $\psi \in \mathcal L_\F^2(0,T)$. Since for any $\bH$-predictable process $\varphi$, the process
$\I_{(0,t]}(s) \varphi_s$, with $t \leq T$, 
is $\bH$-predictable,  condition \eqref{eq:orthogcond} 
implies that for every $t \in [0,T]$ and for each $\varphi \in \mathcal M_{\H}^2(0,T)$, we have
$$
\esp{\psi_T \int_0^t \varphi_s \ud M_s}=0.
$$
Then, 
by conditioning with respect to $\F_t$ (note that $\psi$ is an $\bF$-martingale), for every $\varphi \in \mathcal M_{\H}^2(0,T)$, we get
$$
\esp{\psi_t \int_0^t \varphi_s \ud M_s}= \esp{\int_0^t \varphi_s \ud \langle M, \psi \rangle _s }= 0 \quad \forall t \in [0,T].
$$
From this last equality, we can argue that in the case of full information, i.e., when $\H_t=\F_t$, for each $t \in [0,T]$, condition \eqref{eq:orthogcond} is equivalent to the strong orthogonality condition between $\psi$ and $M$ (see e.g. Lemma 2 and Theorem 36, Chapter IV, page 180 of~\cite{pp} for a rigorous proof). 
\end{remark}

In the sequel, we will say that a square-integrable $\bF$-martingale $O$ is {\em weakly orthogonal} to $M$ if condition \eqref{eq:orthogcond} holds for all processes $\varphi \in \mathcal M_{\H}^2(0,T)$.

\subsection{Existence and Uniqueness}

Our aim is to investigate existence and uniqueness of solutions to the BSDE \eqref{eq:bsde1} with
data $(\xi, f, \bH)$ driven by the general martingale $M$ in the sense of Definition \ref{def:solBSDE}. The case $f\equiv 0$ in \eqref{eq:bsde1}, has been studied in~\cite{ccr},
where 
a key role is played by the Galtchouk-Kunita-Watanabe decomposition under partial information that we recall here for reader's convenience. 
\begin{proposition} \label{prop:GKW}
Let $\xi \in L^2(\Omega,\F_T,\P;\R)$. There exists a unique decomposition of the form
\begin{equation} \label{eq:GKW}
\xi = U_0 + \int_0^T H_t^\H\ud M_t + A_T, \quad \P-\mbox{a.s.},
\end{equation}
where $U_0 \in L^2(\Omega,\F_0,\P;\R)$, $H^\H=(H_t^\H)_{0 \leq t \leq T} \in \mathcal M_{\H}^2(0,T)$ and $A = (A_t)_{0 \leq t \leq T} \in  \mathcal L_\F^2(0,T)$ weakly orthogonal to $M$.  
\end{proposition}
\noindent Inspired by~\cite{bdm02}, we make the following assumption on the $\bF$-predictable quadratic variation $\langle M\rangle$ of $M$.
\begin{ass} \label{ass:bracket}
There exists a deterministic function $\rho:\R_+ \rightarrow \R_+$ with $\rho(0^+)=0$ such that, $\P$-a.s.,
$$
\langle M\rangle_t - \langle M\rangle_s \leq \rho(t-s), \quad \forall 0 \leq s \leq t \leq T.
$$
\end{ass}

\begin{example}
On the probability space $(\Omega,\F,\P)$ let us consider a standard Brownian motion $W$ and an independent Poisson random measure $N(\ud \zeta, \ud t)$ on $ Z \times [0,T] $ with non-negative intensity $\nu(\ud \zeta)\ud t$, where $\nu(\ud \zeta)$ is a $\sigma$-finite measure on a measurable space $(Z, \mathcal Z)$. Denote by $\tilde N$ the corresponding compensated measure defined by
\begin{equation*} \label{def:cm}
\tilde N(\ud \zeta, \ud t)=N(\ud \zeta, \ud t)-\nu(\ud \zeta)\ud t.
\end{equation*}
\noindent 
Let $M$ be given by
\begin{equation*} \label{j-d}
M_t=M_0+  \int_0^t \bar \sigma_s \ud W_s+ \int_0^t \int_Z \bar K(\zeta;s)\tilde N(\ud \zeta, \ud s), \quad t \in [0,T],
\end{equation*}
with $\bar \sigma=(\bar \sigma_t)_{0 \leq t \leq T}$ and $\bar K=(\bar K(\cdot;t))_{0 \leq t \leq T}$ being $\R$-valued, $\bF$-adapted and  $\bF$-predictable  processes respectively, and satisfying 
$$
\esp{\int_0^T \bar \sigma^2_s \ud s +  \int_0^T \int_{Z}\bar K^2(\zeta;s)\nu(\ud\zeta) \ud s}<\infty .
$$
\noindent 
Then, $M$ is a square-integrable  $\bF$-martingale with $\bF$-predictable
quadratic variation process $\langle M\rangle$ given by 
$$
\langle M\rangle_t  =  \int_0^t \left (\bar \sigma^2_s+\int_{Z}\bar K^2(\zeta;s)\nu(\ud\zeta)\right) \ud s, \quad t \in [0,T].
 $$
If in addition we assume that there exists a positive constant $\bar C$ such that 
\begin{equation} \label{SI} \bar \sigma^2_t+\int_{Z}\bar K^2(\zeta;t)\nu(\ud\zeta)\leq \bar C \quad \ud \P \times \ud t-a.e. \end{equation}
then Assumption \ref{ass:bracket} is fulfilled with $ \rho(t-s) = \bar C(t-s)$, with $0 \leq s \leq t \leq T$.

\noindent  Let us observe that in particular condition (\ref{SI}) is satisfied if both processes $\bar \sigma$ and $\bar K$ are bounded and 
$\nu(\{\zeta \in Z:\bar K(\zeta;t) \neq 0\}) < \infty$ for every $t \in [0,T]$.
\end{example}

\noindent We start with the following lemma.
\begin{lemma} \label{lem:existence}
Let Assumption \ref{ass:driver} hold  and assume that $\langle M \rangle_T \leq C(T)$ $\P$-a.s., where $C(T)$ is a positive constant depending on $T$. Let $(U,V)=(U_t,V_t)_{0 \leq t \leq T} \in \mathcal S_\F^2(0,T) \times \mathcal M_{\H}^2(0,T)$. Then the BSDE
\begin{equation} \label{eq:bsdeUV}
Y_t=\xi + \int_t^T f(s,U_{s-},V_s)\ud \langle M\rangle_s - \int_t^T Z_s \ud M_s -(O_T - O_t), \quad 0 \leq t \leq T,
\end{equation}
has a solution with data $(\xi,f,\bH)$ under partial information in the sense of Definition \ref{def:solBSDE}.
\end{lemma}

\begin{proof}
Firstly, we set
\begin{equation*} \label{eq:defY}
Y_t=\condespf{\xi + \int_t^T f(s,U_{s-},V_s) \ud \langle M\rangle_s}, \quad  t \in [0,T].
\end{equation*}
Here $Y$ is a càdlàg $\bF$-adapted process and moreover
$$
|Y_t| \leq m_t:=\condespf{|\xi| + \int_0^T |f(s,U_{s-},V_s)| \ud \langle M\rangle_s}, \quad  t \in [0,T],
$$
where $m=(m_t)_{0 \leq t \leq T}$ is a square-integrable $\bF$-martingale. 
Thus, Doob's inequality and Jensen's inequality yield 
\begin{align*}
\esp{\sup_{0\leq t\leq T}|Y_t|^2} & \leq \esp{\sup_{0\leq t\leq T}|m_t|^2}	\\
& \leq 4 \sup_{0\leq t\leq T}\esp{|m_t|^2}=4\esp{\left(|\xi| + \int_0^T |f(s,U_{s-},V_s)| \ud \langle M\rangle_s\right)^2}\\
& \leq 8\esp{|\xi|^2} + 8\esp{\left(\int_0^T |f(s,U_{s-},V_s)| \ud \langle M\rangle_s\right)^2}.
\end{align*}
By Cauchy-Schwarz inequality and boundedness of $ \langle M\rangle$, we get
\begin{align*}
\esp{\sup_{0\leq t\leq T}|Y_t|^2} & \leq 8\esp{|\xi|^2} + 8 C(T)\esp{\int_0^T |f(s,U_{s-},V_s)|^2 \ud \langle M\rangle_s}.
\end{align*}
Finally, by Assumption \ref{ass:driver}, we have
\begin{align*}
\esp{\sup_{0\leq t\leq T}|Y_t|^2} & \leq 8\esp{|\xi|^2} + 32 C(T)K^2\esp{\int_0^T (|U_{s-}|^2 + |V_s|^2) \ud \langle M\rangle_s}\\
& \quad \quad + 16 C(T)\esp{\int_0^T |f(s,0,0)|^2 \ud \langle M\rangle_s}.
\end{align*}
The right-hand side of  previous inequality is finite in view of hypotheses on $(U,V)$, Assumptions \ref{ass:driver} and \ref{ass:bracket}. Hence, $Y \in  \mathcal S_\F^2(0,T)$.\\
\noindent By Proposition \ref{prop:GKW}, the square-integrable $\F_T$-measurable random variable 
$$
\xi + \int_0^T f(s,U_{s-},V_s) \ud \langle M\rangle_s
$$ 
admits a unique Galtchouk-Kunita-Watanabe decomposition under partial information. 
Setting $Z_t=H_t^\H$ and $O_t=A_t$ for every $t \in [0,T]$, see \eqref{eq:GKW}, this
ensures uniqueness of the process $Z \in \mathcal M_\H^2(0,T)$ and of the process $O \in \mathcal L_\F^2(0,T)$ satisfying the orthogonality condition \eqref{eq:orthogcond}, which verify the BSDE \eqref{eq:bsdeUV}.\\
Indeed, taking the conditional expectation with respect to $\F_t$ yields the following identity:
\begin{align}
\condespf{\xi + \int_0^T f(s,U_{s-},V_s) \ud \langle M\rangle_s} & = \condespf{U_0 + \int_0^T H_s^\H\ud M_s + A_T} \nonumber \\
& = U_0 + \int_0^t H_s^\H\ud M_s + A_t \nonumber \\
& = Y_0 + \int_0^t H_s^\H\ud M_s + A_t, \quad  0 \leq t \leq T. \label{eq:rvrepre}
\end{align}
By \eqref{eq:defY} and 
\eqref{eq:rvrepre} we have that
$$
Y_t + \int_0^t f(s,U_{s-},V_s) \ud \langle M\rangle_s = Y_0 + \int_0^t H_s^\H\ud M_s + A_t, \quad 0 \leq t \leq T,
$$
from which we deduce that
$$
Y_t = \xi  + \int_t^T f(s,U_{s-},V_s) \ud \langle M\rangle_s - \int_t^T H_s^\H\ud M_s -(A_T -  A_t), \quad 0 \leq t \leq T.
$$
\end{proof}

\noindent We keep on the study by giving an estimation result.
\begin{proposition} \label{prop:estimation}
Under 
Assumptions \ref{ass:driver} and \ref{ass:bracket}, let $(Y,Z,O)$ $($respectively $(Y^{'},Z^{'},O^{'}))$ be a solution of the BSDE \eqref{eq:bsde1} with data $(\xi, f, \bH)$ $($respectively with data $(\xi^{'}, f, \bH))$ associated to $(U,V) \in \mathcal S_\F^2(0,T) \times \mathcal M_{\H}^2(0,T)$ $($respectively $(U^{'},V^{'}) \in \mathcal S_\F^2(0,T) \times \mathcal M_{\H}^2(0,T))$. Then, 
for each $0 \leq u \leq v \leq T$, we have
\begin{equation} \label{eq:estimation}
\begin{split}
 & \esp{\sup_{u \leq t \leq v}|\delta Y_t|^2 + \int_u^v|\delta Z_s|^2 \ud \langle M\rangle_s + \langle \delta O \rangle_v - \langle \delta O \rangle_u}\\ 		
& \qquad \qquad \qquad \qquad \leq 42\esp{|\delta Y_v|^2} + C(v-u)\esp{\sup_{u \leq t \leq v}|\delta U_t|^2 + \int_u^v|\delta V_s|^2 \ud \langle M\rangle_s},
\end{split}
\end{equation}
where $C(r)=42K^2\max\{\rho^2(r), \rho(r)\}$ and $\delta Y$ stands for $Y-Y^{'}$ and so on.
\end{proposition}

\begin{proof}
For reader's convenience, here we provide
briefly the proof of \eqref{eq:estimation}. It is formally analogous to the one of Proposition 7 in~\cite{bdm02}. The difference is due to the orthogonality condition we consider in this framework.
We start by the following equation: for every $t \in [0,v] \subseteq [0,T]$, set
\begin{equation} \label{eq:deltaY}
\delta Y_t = \delta Y_v + \int_t^v\left(f(s,U_{s-},V_s) - f(s,U_{s-}^{'},V_s^{'})\right)\ud \langle M\rangle_s - \int_t^v \delta Z_s \ud M_s -(\delta O_v - \delta O_t).
\end{equation}
Since $f$ is $K$-Lipschitz in virtue of Assumption \ref{ass:driver},  for any $t \in [0,v]$ we have
$$
|\delta Y_t| \leq \condespf{|\delta Y_v| + K \int_t^v\left(|\delta U_{s-}| + |\delta V_s|\right) \ud \langle M \rangle_s} \leq \tilde m_t,
$$
where $\tilde m=(\tilde m_t)_{0 \leq t \leq T}$, defined by $\tilde m_t:= \condespf{|\delta Y_v| + K \int_u^v\left(|\delta U_{s-}| + |\delta V_s|\right) \ud \langle M \rangle_s}$ for each $t \in [0,T]$, is a square-integrable $\bF$-martingale. Doob's inequality 
gives
\begin{align}
\esp{\sup_{u\leq t\leq v}|\delta Y_t|^2} & \leq \esp{\sup_{u\leq t\leq v}|\tilde m_t|^2} \leq 4 \sup_{u\leq t\leq v}\esp{|\tilde m_t|^2}\nonumber \\
& \leq 4 \esp{\left(|\delta Y_v| + K \int_u^v\left(|\delta U_{s-}| + |\delta V_s|\right) \ud \langle M \rangle_s\right)^2}. \label{eq:Doob}
\end{align}
Furthermore, since $\delta O$ satisfies the orthogonality condition \eqref{eq:orthogcond}, it is easy to check that
$$
\esp{\left(\delta O_v - \delta O_u\right)\int_u^v \delta Z_s \ud M_s}=0, \quad 0 \leq u\leq v \leq T;
$$
then
\begin{equation} \label{eq:omorthog}
\esp{\int_u^u|\delta Z_s|^2 \ud \langle M\rangle_s + \langle \delta O\rangle_v - \langle \delta O\rangle_u}=\esp{\left|\int_u^v \delta Z_s \ud M_s + \delta O_v - \delta O_u\right|^2}.
\end{equation}
Hence, taking \eqref{eq:deltaY} into account we derive
$$
\int_u^v \delta Z_s \ud M_s + \delta O_v - \delta O_u = \delta Y_v - \delta Y_u + \int_u^v \left(f(s,U_{s-},V_s)-f(s,U_{s-}^{'},V_s^{'})\right)\ud \langle M\rangle_s.
$$
Using the fact that $f$ is $K$-Lipschitz in virtue of Assumption \ref{ass:driver}, we obtain
$$
\left|\int_u^v \delta Z_s \ud M_s + \delta O_v - \delta O_u\right| \leq |\delta Y_v| + \sup_{u \leq t \leq v}|\delta Y_t|+K \int_u^v\left(|\delta U_{s-}| + |\delta V_s|\right) \ud \langle M \rangle_s.
$$
Since \eqref{eq:Doob} also holds for $\esp{\left(\sup_{u\leq t\leq v}|\delta Y_t|\right)^2}$, from the estimate \eqref{eq:Doob} and relationship \eqref{eq:omorthog}, we get 
\begin{align*}
 & \esp{\sup_{u \leq t \leq v}|\delta Y_t|^2 + \int_u^v|\delta Z_s|^2 \ud \langle M\rangle_s + \langle \delta O \rangle_v - \langle \delta O \rangle_u}\\ 		
& \qquad \qquad \qquad \qquad \leq 14 \esp{\left(|\delta Y_v|+K \int_u^v\left(|\delta U_{s-}| + |\delta V_s|\right) \ud \langle M \rangle_s\right)^2}.
\end{align*}
Cauchy-Schwarz inequality together with Assumption \ref{ass:bracket} lead to the estimate
\begin{equation*}
\begin{split}
 & \esp{\sup_{u \leq t \leq v}|\delta Y_t|^2 + \int_u^v|\delta Z_s|^2 \ud \langle M\rangle_s + \langle \delta O \rangle_v - \langle \delta O \rangle_u}\\ 		
& \qquad \qquad \qquad \qquad \leq 42\esp{|\delta Y_v|^2} + C(v-u)\esp{\sup_{u \leq t \leq v}|\delta U_t|^2 + \int_u^v|\delta V_s|^2 \ud \langle M\rangle_s},
\end{split}
\end{equation*}
with $C(v-u)=42K^2\max\{\rho^2(v-u), \rho(v-u)\}$.
\end{proof}
\noindent Note that, since $\lim_{r \to 0^+}\rho(r)=0$ by Assumption \ref{ass:bracket}, there exists $r_0 \in (0,T)$ such that $42K^2\max\{\rho^2(v-u), \rho(v-u)\} \leq \frac{1}{6}$ as soon as $r \leq r_0$. Similarly to~\cite{bdm02}, we introduce the following norm on $\mathcal S_\F^2(0,T) \times \mathcal M_\H^2(0,T) \times \mathcal L_\F^2(0,T)$:
$$
\|(Y,Z,O)\|_p^2:=\sum_{k=0}^{\hat m-1}(5 \cdot 42)^k\esp{\sup_{I_k}|Y_t|^2 + \int_{I_k}|Z_s|^2 \ud \langle M\rangle_s + \langle O\rangle_{\frac{(k+1)T}{\hat m}} - \langle O\rangle_{\frac{kT}{\hat m}}},
$$
where $\hat m=[T/r_0]+1$ is fixed and $I_k=[kT/\hat m,(k+1)T/\hat m]$, for $0 \leq k \leq \hat m-1$, are $\hat m$ intervals that constitute a regular partition of $[0,T]$. 
This norm is equivalent to the classical one since we have
$$
\|(Y,Z,O)\|^2:=\|Y\|_{\mathcal S^2}^2 + \|Z\|_{\mathcal M^2}^2 + \|O\|_{\mathcal L^2}^2 \leq \|(Y,Z,O)\|_p^2 \leq \hat m(5 \cdot 42)^{\hat m-1}\|(Y,Z,O)\|^2.
$$
Thanks to the estimate of Proposition \ref{prop:estimation} and a straightforward computation, we can show that if $(Y,Z,O)$ and $(Y^{'},Z^{'},O^{'})$ are the solutions to the BSDE \eqref{eq:bsdeUV} with $(\xi,U,V)$ and $(\xi^{'},U^{'},V^{'})$ respectively, then we have
$$
\|(\delta Y,\delta Z,\delta O)\|_p^2 \leq \frac{1}{5} \|(\delta Y,\delta Z,\delta O)\|_p^2+ C(T/\hat m)\|(\delta U,\delta V,0)\|_p^2.
$$
Hence
\begin{equation} \label{eq:estimate}
\|(\delta Y,\delta Z,\delta O)\|_p^2 \leq \frac{1}{4} \|(\delta U,\delta V,0)\|_p^2.
\end{equation} 
\begin{theorem} \label{th:ex-uniq}
Let 
Assumptions \ref{ass:driver} and \ref{ass:bracket} hold.
Given data $(\xi, f, \bH)$, there exists a unique triplet $(Y,Z,O)$ which solves the
BSDE \eqref{eq:bsde1} under partial information 
in the sense of Definition \ref{def:solBSDE}.
\end{theorem}

\begin{proof}
The idea is to use a fixed point argument. Let us consider the application $\Phi$ from $\mathcal S_\F^2(0,T) \times \mathcal M_\H^2(0,T) \times \mathcal L_\F^2(0,T)$ into itself which is defined by setting $\Phi(U,V,L)=(Y,Z,O)$ where $(Y,Z,O)$ is the solution to the BSDE \eqref{eq:bsdeUV}. Note that $L$ does not appear and the application is well-defined thanks to Lemma \ref{lem:existence} and since estimate \eqref{eq:estimate} ensures the existence of a unique solution, in the space $\mathcal S_\F^2(0,T) \times \mathcal M_\H^2(0,T) \times \mathcal L_\F^2(0,T)$, to the BSDE \eqref{eq:bsdeUV}, once the pair $(U,V) \in \mathcal S_\F^2(0,T) \times \mathcal M_\H^2(0,T)$ is fixed.\\
Indeed, the estimate \eqref{eq:estimate} says that $\Phi$ is a contraction with constant $\frac{1}{2}$ if we use the equivalent norm $\|\cdot\|_p$ instead of the the classical one $\|\cdot\|$ on the Banach space $\mathcal S_\F^2(0,T) \times \mathcal M_\H^2(0,T) \times \mathcal L_\F^2(0,T)$.
\end{proof}

Since it will be useful in the sequel, we recall for reader's convenience the definition of $\bH$-predictable dual projection.
\begin{definition}
Let $G=(G_t)_{0\leq t \leq T}$ be a càdlàg $\bF$-adapted process of integrable variation. The $\bH$-predictable dual projection of $G$ is the unique $\bH$-predictable process $G^\bH=(G_t^\bH)_{0\leq t \leq T}$ of integrable variation such that
$$
\esp{\int_0^T \varphi_s \ud G_t^\bH} = \esp{\int_0^T \varphi_s \ud G_t},
$$
for every $\bH$-predictable (bounded) process $\varphi$.
\end{definition}
It is possible to show that BSDEs under partial information can be reduced to full information problems, which however are not described by a BSDE, unless the driver does not depend on $z$ (see Proposition \ref {nuova1} below). 
More precisely, we have the following result.
\begin{proposition} \label{nuova}
Let $(\tilde Y,\tilde Z, \tilde O)\in \mathcal S_\F^2(0,T) \times \mathcal M_\F^2(0,T) \times \mathcal L_\F^2(0,T)$ be a solution to the problem under complete information
\begin{equation} \label{eq:full}
\tilde Y_t = \xi + \int_t^Tf(s,\tilde Y_{s-}, \hat Z_s)\ud \langle M\rangle_s - \int_t^T \tilde Z_s \ud M_s - (\tilde O_T - \tilde O_t), \quad 0 \leq t \leq T,
\end{equation}
where $\tilde O$ is strongly orthogonal to $M$ and 
$$ 
\hat Z_t=  \frac{\ud L^\bH_t}{\ud \langle M\rangle^\bH_t} \quad L_t:=\int_0^t \tilde Z_s \ud \langle M \rangle_s, \quad 0 \leq t \leq T.
$$
Then  the triplet
$$
(Y,Z,O)=\left(\tilde Y, \hat Z, \tilde O + B\right),
$$
where $B= \int( \tilde Z_s -  \hat Z_s)   \ud M_s $ is a square-integrable $\bF$-martingale weakly orthogonal to $M$, 
is a solution to the BSDE \eqref{eq:bsde1} under partial information.

\end{proposition}

\begin{proof}
First let us observe that by Proposition 4.8 of~\cite{ccr}, $L^\bH:= (\int\tilde Z_s \ud \langle M \rangle_s)^\bH$ is   absolutely continuous with respect to $ \langle M\rangle^\bH$, hence  $\hat Z$ is well defined.  
By \eqref{eq:full} we get
$$
\tilde Y_t = \xi + \int_t^Tf(s,\tilde Y_{s-}, \hat Z_s)\ud \langle M\rangle_s - \int_t^T \hat Z_s \ud M_s - \int _t^T ( \tilde Z_s -  \hat Z_s)   \ud M_s  - (\tilde O_T - \tilde O_t), \quad 0 \leq t \leq T.
$$
Set $ B_ t = \int _0^t ( \tilde Z_s -  \hat Z_s)   \ud M_s $,  for each $t \in [0,T]$. It is sufficient to prove that $B=(B_t)_{0 \leq t \leq T}$ is a square-integrable $\bF$-martingale weakly orthogonal to $M$, that is, for every $\varphi \in \mathcal M_\H^2(0,T)$ we have
$$
 \esp{ \int_0^T \varphi_s \ud M_s   \int_0^T ( \tilde Z_s -  \hat Z_s)   \ud M_s } =  \esp{ \int_0^T \varphi_s ( \tilde Z_s -  \hat Z_s)  \langle M\rangle_s } = 0.
$$
In fact
$$\esp{ \int_0^T \varphi_s \tilde Z_s \langle M\rangle_s } = \esp{ \int_0^T \varphi_s ( \tilde Z \langle M\rangle)_s ^\bH} = \esp{ \int_0^T \varphi_s \hat Z_s \langle M\rangle_s ^\bH} =  \esp{ \int_0^T \varphi_s \hat Z_s \langle M\rangle_s}. $$ 

Finally, let us observe that the above equality is fulfilled for any $\bH$-predictable process $\varphi$. Hence we can choose   $\varphi = \hat Z$
and get
$$ \esp{\int_0^T|\hat Z_s|^2\ud \langle M\rangle_s} = \esp{\int_0^T\hat Z_s   \tilde Z_s\ud \langle M\rangle_s}. $$
Then, by Cauchy-Schwarz inequality we obtain
$$ \esp{\int_0^T|\hat Z_s|^2\ud \langle M\rangle_s} \leq  \left\{ \esp{\int_0^T |\tilde Z_s|^2 \ud \langle M\rangle_s} \right\}^{1\over 2}
 \left \{\esp{\int_0^T|\hat Z_s|^2 \ud \langle M\rangle_s} \right \}^{1\over 2}$$
  which in turn implies
$$ \esp{\int_0^T|\hat Z_s|^2\ud \langle M\rangle_s} \leq \esp{\int_0^T|\tilde Z_s|^2\ud \langle M\rangle_s} <  \infty.$$
\end{proof}
\begin{remark}
As shown in Section 4 of~\cite{ccr}, in some cases it is possible to compute explicitly the Radon-Nikod\'ym derivative of $L^\bH$ with respect to $\langle M\rangle^\bH$ that characterizes the component $Z$ of solution. For instance, if $\langle M\rangle$ is of the form
$$
\langle M\rangle_t=\int_0^t a_s\ud G_s, \quad t \in [0,T]
$$
for some $\bF$-predictable process $a=(a_t)_{0\leq t \leq T}$ and an increasing deterministic function $G$, then
$$
Z_t=\frac{{}^p(\tilde Z_t a_t)}{{}^pa_t},\quad t \in [0,T],
$$ 
where the notation ${}^p X$ refers to the $\bH$-predictable projection of the process $X$. Another meaningful example is given by assuming $\langle M \rangle$ to be $\bH$-predictable. In this case, we have
$$
 Z_t={}^p \tilde Z_t, \quad t \in [0,T].
$$
\end{remark}

\subsection{ Some special cases}
We are now in the position to provide an explicit characterization of the solution to the BSDE \eqref{eq:bsde1} under partial information in terms of the one related to the corresponding BSDE in the case of full information when the driver $f$ does not depend on $z$. 
 
\begin{proposition}\label{nuova1}
Suppose that the driver $f$ is independent of $z$ and let $(\tilde Y,\tilde Z, \tilde O) \in S_\F^2(0,T) \times \mathcal M_\F^2(0,T) \times \mathcal L_\F^2(0,T)$ be a solution to the following BSDE  under complete information 
\begin{equation} \label{eq:bsdeF1}
\tilde Y_t = \xi + \int_t^Tf(s,\tilde Y_{s-})\ud \langle M\rangle_s - \int_t^T \tilde Z_s \ud M_s - (\tilde O_T - \tilde O_t), \quad 0 \leq t \leq T,
\end{equation}
where $ \tilde O=(\tilde O_t)_{0 \leq t \leq T}$ is a square-integrable $\bF$-martingale strongly orthogonal to $M$.
Set $L_t:=\int_0^t \tilde Z_s \ud \langle M \rangle_s$ for each $t\in[0,T]$. Then, the triplet
$$
(Y,Z,O)=\left(\tilde Y, \frac{\ud L^\bH}{\ud \langle M\rangle^\bH},\tilde O + B\right),
$$
where $B= \int( \tilde Z_s -  Z_s)   \ud M_s $ is a square-integrable $\bF$-martingale weakly orthogonal to $M$, 
is a solution to the BSDE 
\begin{equation} \label{eq:bsdeH}
Y_t = \xi + \int_t^Tf(s,Y_{s-})\ud \langle M\rangle_s - \int_t^TZ_s \ud M_s - (O_T - O_t), \quad 0 \leq t \leq T,
\end{equation}
under partial information in the sense of Definition \ref{def:solBSDE}. 
\end{proposition}

\begin{proof}
It is a direct consequence of Proposition \ref{nuova}.
\end{proof}
We conclude this subsection by applying Proposition \ref{nuova} to provide existence of the solution to a BSDE under partial information in the special case where ${\mathcal H}_t = {\mathcal F}_{(t-\tau)^+}$ for each $t \in [0,T]$,
with $\tau \in (0,T)$ being a fixed delay, the driver does not depend on $y$ and $\langle M\rangle$ and $f(\cdot,\cdot,z)$ are $\bH$-predictable processes. This approach allows us to weaken the assumptions required in Theorem \ref{th:ex-uniq}. More precisely, 
we  just require that $f$ satisfies a sublinear growth condition in   $z$. 

Without loss of generality, we take $T = \tau N$, with $N \in \N$. We will solve backwardly equation (\ref{eq:full}) on each interval  $I_j= [(j-1)\tau, j\tau]$, $j\in\{1,\ldots,N\}$. To this aim we need a preliminary Lemma.
 
\begin{lemma} \label{Lj}
Let ${\mathcal H}_t = {\mathcal F}_{(t-\tau)^+}$, for each $t \in [0,T]$, with $\tau \in (0,T)$ being a fixed delay and assume that $\langle M \rangle_T \leq C(T)$ $\P$-a.s., where $C(T)$ is a positive constant depending on $T$. Let $\langle M\rangle$  and $f(\cdot,\cdot,z)$ be $\bH$-predictable and $f$ to satisfy a sublinear growth  condition with respect to $z$ uniformly in $(\omega,t)$, i.e.

$\exists$ $C\geq 0$ such that $\forall z \in \R,$ $
|f(\omega,t,z)|^2 \leq C(1 + |z|^2)  \quad (\P \otimes \langle M\rangle)-\mbox{a.e.\ on}\ \Omega \times [0,T];
$
 
Let $\xi^j \in  L^2(\Omega,\F_{j\tau},\P;\R)$. Then there exists a solution  $(\tilde Y^j,\tilde Z^j, \tilde O^j)\in \mathcal S_\F^2((j-1)\tau, j\tau) \times \mathcal M_\F^2((j-1)\tau, j\tau) \times \mathcal L_\F^2((j-1)\tau, j\tau)$ to the problem under complete information
\begin{equation} \label{eqj:full}
\tilde Y^j_t = \xi^j + \int_t^{j\tau}f(s, {}^p \tilde Z^j_s)\ud \langle M\rangle_s - \int_t^{j\tau} \tilde Z^j_s \ud M_s - (\tilde O^j_{j\tau} - \tilde O^j_t), \quad (j-1)\tau \leq t \leq j\tau,
\end{equation}
where $\tilde O^j$ is strongly orthogonal to $M$ and $ {}^p \tilde Z^j$ denotes the $\bH$-predictable projection of  $\tilde Z^j$, that is, 
$ {}^p \tilde Z^j_t = \esp{ \tilde Z^j_t | \mathcal H_{t^-}}$, for every $t \in [0,T]$. 
\end{lemma}
 
 \begin{proof}
 According to the Galtchouk-Kunita-Watanabe decomposition of $ \xi^j$ under full information, there exists $\tilde Z^j \in \mathcal M_\F^2((j-1)\tau, j\tau)$ such that
 \begin{equation} \label{eqj}
 \xi^j= \esp{\xi^j | {\mathcal F}_{(j-1) \tau}} +  \int_{(j-1) \tau}^{j\tau} \tilde Z^j_s \ud M_s + \left(\tilde O^j_{j\tau} - \tilde O^j_{(j-1) \tau}\right),
  \end{equation} 
 where $\tilde O^j \in  L_\F^2((j-1)\tau, j\tau)$ is strongly orthogonal to $M$. For every $t \in [(j-1)\tau, j\tau]$, we set 
 \begin{equation}  \label{eqj1}
  Y^j_t = \esp{\xi^j | {\mathcal F}_t} +  \int_t^{j\tau}f(s, {}^p \tilde Z^j_s) \ud \langle M\rangle_s.
  \end{equation} 
  Let us observe that $Y^j \in S_\F^2((j-1)\tau, j\tau)$.  In fact,  since  $\int_t^{j\tau}f(s, {}^p \tilde Z^j_s) \ud \langle M\rangle_s$ is ${\mathcal F}_{(j-1) \tau}$-measurable $Y^j$ turns out to be 
 $\bF$-adapted. By the sublinear growth condition on $f$, Jensen's inequality and the property of the $\bH$-predictable projection, we get
 \begin{align*}
 \esp{ \int_{(j-1)\tau}^{j\tau}|f(s, {}^p \tilde Z^j_s)|^2 \ud \langle M\rangle_s} & \leq  \esp{\int_t^{j\tau} C( 1+ |{}^p \tilde Z^j_s|^2) \ud \langle M\rangle_s} \\
& \leq \esp{  \int_{(j-1)\tau}^{j\tau} C ( 1+ {}^p (| \tilde Z^j|^2_s)) \ud \langle M\rangle_s} \\
& = C \hskip 1mm
 \esp{ \langle M\rangle_{j\tau} -  \langle M\rangle_{(j-1)\tau}  +\int_{(j-1)\tau}^{j\tau}  |\tilde Z^j|^2_s \ud \langle M\rangle_s} < \infty,
\end{align*}
and by performing the same computation as in the proof of Lemma \ref{lem:existence}, we finally obtain
 $$ \esp{\sup_{(j-1)\tau\leq t\leq j\tau }|Y^j_t|^2}  \leq 8\esp{|\xi^j|^2} + 8 C(T) \esp{\int_{(j-1)\tau}^{j\tau} |f(s, {}^p \tilde Z^j_s)|^2 \ud \langle M\rangle_s} < \infty. $$
 We now take the conditional expectation with respect to $ {\mathcal F}_t$ in \eqref{eqj} and for each $ t \in [(j-1)\tau, j\tau]$ we obtain
 \begin{equation}  \label{eqj2}
 \esp{\xi^j | {\mathcal F}_t} -  \esp{\xi^j | {\mathcal F}_{(j-1) \tau}}  =   \int_{(j-1) \tau}^{t} \tilde Z^j_s \ud M_s + \left(\tilde O^j_{t} - \tilde O^j_{(j-1) \tau}\right).
 \end{equation}
 At this stage, subtracting \eqref{eqj2} and \eqref{eqj} yields
 $$\esp{\xi^j | {\mathcal F}_t} - \xi^j = - \int_{(j-1) \tau}^{t} \tilde Z^j_s \ud M_s - \left(\tilde O^j_{j \tau t} - \tilde O^j_{t}\right)$$
and using (\ref{eqj1})  we get 
 $$  Y^j_t  -  \int_t^{j\tau}f(s, {}^p \tilde Z^j_s) \ud \langle M\rangle_s - \xi^j = - \int_{(j-1) \tau}^{t} \tilde Z^j_s \ud M_s - \left(\tilde O^j_{j \tau } - \tilde O^j_{t}\right),$$
 which concludes the proof.
\end{proof} 
We are now in the position to state the following result.
 \begin{proposition}\label{2nuova}
Let ${\mathcal H}_t = {\mathcal F}_{(t-\tau)^+}$, for each $t \in [0,T]$, with $\tau \in (0,T)$ being a fixed delay and assume that $\langle M \rangle_T \leq C(T)$ $\P$-a.s., where $C(T)$ is a positive constant depending on $T$. Let  $\xi \in L^2(\Omega,\F_{T},\P;\R)$, $\langle M\rangle$  and $f(\cdot,\cdot,z)$ be $\bH$-predictable and $f$ to satisfy a sublinear growth  condition with respect to $z$ uniformly in $(\omega,t)$, i.e.

$\exists$ $C\geq 0$ such that $\forall z \in \R,$ $
|f(\omega,t,z)|^2 \leq C(1 + |z|^2)  \quad (\P \otimes \langle M\rangle)-\mbox{a.e.\ on}\ \Omega \times [0,T].$

Then, there exists a solution $( Y, Z,  O) \in S_\F^2(0,T) \times \mathcal M_\H^2(0,T) \times \mathcal L_\F^2(0,T)$ to the BSDE under restricted information
\begin{equation} \label{lasteq}
Y_t= \xi  +  \int_t^{T}f(s, Z_s) \ud \langle M\rangle_s   - \int_t^T Z_s \ud M_s - (\tilde O_T - \tilde O_{t}).
\end{equation}
\end{proposition}
 
\begin{proof}
We apply Lemma \ref{Lj}. Set  $\xi^N = \xi$, and $\xi^j = \tilde Y^{j+1}_{j\tau}$, $j=1, 2... N-1$,  where  $(\tilde Y^j,\tilde Z^j, \tilde O^j)\in \mathcal S_\F^2((j-1)\tau, j\tau) \times \mathcal M_\F^2((j-1)\tau, j\tau) \times \mathcal L_\F^2((j-1)\tau, j\tau)$ is the solution of the problem under complete information \eqref{eqj:full}. 

Set $\tilde Y_t:= \sum_{j=1}^N Y^j_t \I_{\{ t \in [(j-1)\tau, j\tau)\}}$, $\tilde Z_t := \sum_{j=1}^N Y^j_t \I_{\{ t \in ((j-1)\tau, j\tau]\}}$, $\tilde O_t:= \sum_{j=1}^N O^j_t \I_{\{ t \in [(j-1)\tau, j\tau)\}}$.
Then, we get that  the triplet 
$$(\tilde Y, \tilde Z,  \tilde O)  \in \mathcal S_\F^2(0,T) \times \mathcal M_\F^2(0,T) \times \mathcal L_\F^2(0,T)$$ 
is a solution to the problem under complete information
\begin{equation} \label{eq1:full}
\tilde Y_t = \xi + \int_t^Tf(s, {}^p \tilde Z_s)\ud \langle M\rangle_s - \int_t^T \tilde Z_s \ud M_s - (\tilde O_T - \tilde O_t), \quad 0 \leq t \leq T.
\end{equation}
 
Finally by applying Proposition \ref{nuova},  the triplet $( Y, Z,  O) = (\tilde Y,  {}^p \tilde Z, \tilde O + B) \in \mathcal S_\F^2(0,T) \times \mathcal M_\H^2(0,T) \times \mathcal L_\F^2(0,T)$,
 where $B= \int( \tilde Z_s -  Z_s)   \ud M_s $ is a square-integrable $\bF$-martingale weakly orthogonal to $M$,
solves the BSDE (\ref{lasteq}) under restricted information. 
\end{proof}

In the next section, we will 
apply the existence and uniqueness results obtained for BSDEs to derive the F\"ollmer-Schweizer decomposition in a partial information framework and 
discuss a financial application. More precisely, we will study the hedging problem of a contingent claim in incomplete markets when the underlying price process is given by a general $\bF$-semimartingale and there are restrictions on the available information to traders. 

\section{Local Risk-Minimization under restricted information} \label{LRM}

\noindent Let us fix a probability space $(\Omega, \F, \P)$ endowed with a filtration $\bF :=(\F_t)_{0\leq t\leq T}$ satisfying the usual
conditions of right-continuity and completeness. Here $T >0$ denotes a fixed and finite time horizon; furthermore, we assume that $\F=\F_T$.
We consider a financial market with one riskless asset with (discounted) price
1 and a risky asset whose (discounted) price $S$ is described by an $\R$-valued square-integrable (càdlàg) $\bF$-semimartingale  $S=(S_t)_{0 \leq t \leq T}$ satisfying the so-called structure condition (SC), that is 
\begin{equation} \label{eq:SC}
S_t = S_0 + M_t + \int_0^t \alpha_s \ud \langle M\rangle_s, \quad 0 \leq t \leq T,
\end{equation}
where $M=(M_t)_{0 \leq t \leq T}$ is an $\R$-valued square-integrable (càdlàg) $\bF$-martingale with $M_0=0$ and $\bF$-predictable
quadratic variation process denoted by $\langle M\rangle = (\langle M,M\rangle)_{0 \leq t \leq T}$ and $\alpha=(\alpha_t)_{0 \leq t \leq T}$ is an $\bF$-predictable  process
such that $\esp{\int_0^T| \alpha_t|^2\ud \langle M\rangle_t} < \infty$.
\begin{remark}
It is known that the existence of an equivalent martingale measure for the risky asset price process $S$ implies that $S$ is an $\bF$-semimartingale under the basic measure $\P$. Then, the semimartingale structure for $S$ is a natural assumption in a financial market model which ensures the absence of arbitrage opportunities. If in addition, $S$ has continuous trajectories or has càdlàg paths and the following condition holds
$$
\esp{\sup_{t \in [0,T]}S_t^2} < \infty,
$$
then $S$ satisfies the structure condition (SC), see page 24 of~\cite{as} and Theorem 1 in~\cite{ms95}.   
\end{remark}
In this framework we consider a contingent claim whose payoff is represented by a random variable $\xi \in L^2(\Omega,\F_T,\P;\R)$.
Under the condition that the mean-variance tradeoff process $K=(K_t)_{0 \leq t \leq T}$ defined by
$$
K_t := \int_0^t \alpha^2_s \ud \langle M\rangle_s, \quad \forall t \in [0,T], 
$$
is uniformly bounded in $t$ and $\omega$, 
in Theorem 3.4 of~\cite{ms95} it is proved  that every $\xi \in L^2(\Omega,\F_T,\P;\R)$ admits a strong F\"ollmer-Schweizer decomposition with respect to $S$, that is 
 \begin{equation} \label{F-S}
\xi = \tilde U_0 + \int_0^T \beta_t  \ud S_t + \tilde A_T, \quad \P-\mbox{a.s.},
\end{equation}
where $\tilde U_0 \in L^2(\Omega,\F_0,\P;\R)$, $\beta = (\beta_t)_{0 \leq t \leq T}$ is an $\bF$-predictable process such that the stochastic integral $\int \beta_t \ud S_t$ is well-defined and it is a square-integrable $\bF$-semimartingale and $\tilde A \in \mathcal L_\F^2(0,T)$ is strongly orthogonal to $M$, see \eqref{eq:SC}.
Moreover, it is known that every $\xi \in L^2(\Omega,\F_T,\P;\R)$ admits a decomposition (\ref{F-S}) if and only if there exists a locally risk-minimizing hedging strategy (see e.g.~\cite{fs,s01}) and in addition this decomposition plays an essential role in the variance-minimizing strategy computation (see~\cite{ms94} for further details).\\
Suppose now that the hedger does not have at her/his disposal the full information represented by
$\bF$; her/his strategy must be constructed from less information. This leads to a partial information framework. To describe this mathematically,
we introduce an additional filtration $\bH :=(\H_t)_{0\leq t\leq T}$ satisfying the usual conditions and such that $\H_t \subseteq \F_t$, for every $t \in [0,T]$. 
Thanks to Theorem \ref{th:ex-uniq}, we are now in the position to derive a similar decomposition in a partial information setting. We need the following additional hypothesis.
\begin{ass} \label{ass:alpha}
There exists a constant $\bar K \ge 0$ such that the process $\alpha$ in \eqref{eq:SC} satisfies:
$$
|\alpha_t(\omega)| \leq \bar K, \quad (\P \otimes \langle M\rangle)-\mbox{a.e.\ on}\ \Omega \times [0,T].
$$
\end{ass}

\begin{proposition} \label{th:FS} 
Let
Assumptions \ref{ass:bracket} and \ref{ass:alpha} hold.
Then, every $\xi \in L^2(\Omega,\F_T,\P;\R)$ admits the following decomposition
\begin{equation} \label{F-Sweak}
\xi = \bar U_0 + \int_0^T \beta^\H_t  \ud S_t + A_T, \quad \P-\mbox{a.s.},
\end{equation}
where $\bar U_0 \in L^2(\Omega,\F_0,\P;\R)$, $\beta^\H = (\beta_t^\H)_{0 \leq t \leq T} \in \mathcal M_\H^2(0,T)$ and $A \in \mathcal L_\F^2(0,T)$ is
weakly orthogonal to $M$. 
\end{proposition}

\noindent In the martingale case where $\alpha \equiv 0$ in \eqref{eq:SC}, representation \eqref{F-Sweak} corresponds to the Galtchouk-Kunita-Watanabe decomposition \eqref{eq:GKW} of $\xi$ under partial information. In the general semimartingale case, \eqref{F-Sweak} is referred as the F\"ollmer-Schweizer decomposition of $\xi$ with respect to $S$ under partial information.
\begin{proof}
Let us consider the driver of the BSDE \eqref{eq:bsde1} under partial information  given by $f(t,y,z) =- z \alpha$, where $\alpha$ is the bounded process introduced in \eqref{eq:SC}. Since Assumption \ref{ass:driver} is fulfilled, by Theorem \ref{th:ex-uniq} there exists a unique triplet $(Y,Z,O)$ which solves the equation
\begin{equation}\label{eq:bsdeFS}
Y_t=\xi - \int_t^T Z_s \alpha_s \ud \langle M\rangle_s - \int_t^T Z_s \ud M_s -(O_T - O_t), \quad 0 \leq t \leq T,
\end{equation}
under partial information in the sense of Definition \ref{def:solBSDE}.
Hence 
$$\xi = Y_T = Y_0 +  \int_0^T Z_s \alpha_s \ud \langle M\rangle_s + \int_0^T Z_s \ud M_s  + O_T =  Y_0 +  \int_0^T Z_s \ud S_s + O_T$$
and we obtain decomposition (\ref{F-Sweak}) by setting $\bar U_0=Y_0$, $ \beta^\H_t = Z_t$ and $A_t=O_t$, for every $t \in [0,T]$.
 \end{proof}

\begin{remark}
Note that if $Y$ represents the wealth that satisfies the replication constraint $Y_T=\xi$ $\P$-a.s., the triplet $(Y,\beta^\H,A)$ may be interpreted as the {\em nonadjusted hedging strategy} against $\xi$. Clearly, the 
self-financing condition of the strategy is no longer ensured due to the 
presence of the cost $A$, see~\cite{kpq} for further details. 
\end{remark}

\noindent We now study the relationship between the F\"ollmer-Schweizer decomposition of a contingent claim $\xi \in L^2(\Omega,\F_T,\P;\R)$
under partial information and the
existence of a {\em locally risk-minimizing} strategy in a partial information
framework. In the sequel, we will suppose that Assumptions \ref{ass:bracket} and \ref{ass:alpha} are in force.\\
In this setting, the amount $\theta=(\theta_t)_{0 \leq t \leq T}$ invested by the agent in the risky asset has to be adapted to the information flow $\bH$ and such that the stochastic integral $\int \theta_u \ud S_u$ turns out to be a square-integrable $\bF$-semimartingale. By Assumption \ref{ass:bracket} and boundedness of $\alpha$, we will look at the class of processes $\theta$ such that $\theta \in \mathcal M_\H^2(0,T)$. Indeed,
\begin{align*}
\esp{\int_0^T \theta_s^2 \ud \langle M\rangle_s + \left(\int_0^T|\theta_s||\alpha_s| \ud  \langle M\rangle_s\right)^2} & \leq \esp{\int_0^T \theta_s^2 \ud \langle M\rangle_s + \bar{K}^2\left(\int_0^T|\theta_s|\ud \langle M\rangle_s\right)^2}\\
& \leq \esp{\int_0^T \theta_s^2 \ud \langle M\rangle_s + \bar{K}^2 \int_0^T \theta_s^2 \ud \langle M\rangle_s \cdot \langle M\rangle_T}\\
& \leq \left(1 + \bar{K}^2\rho(T)\right) \esp{\int_0^T \theta_s^2 \ud \langle M\rangle_s}.
\end{align*}
Clearly, in this case Assumption \ref{ass:bracket} can be weakened by requiring that $\langle M\rangle_T \leq C(T)$ $\P$-a.s., for a positive constant $C(T)$ depending on $T$.
\begin{definition}
An $(\bH,\bF)$-{\em admissible strategy} is a pair $\Psi=(\theta,\eta)$ where $\theta \in \mathcal M_\H^2(0,T)$ and $\eta=(\eta_t)_{0 \leq t \leq T}$ is a real-valued $\bF$-adapted process such that the value process $V(\Psi):=\theta S + \eta$ is right-continuous and satisfies $V_t(\Psi) \in  L^2(\Omega,\F_t,\P;\R)$ for each $t \in [0,T]$.
\end{definition}

\begin{remark}
We assume that the agent has at her/his disposal the information flow $\bH$ about trading in the risky asset while a complete information about trading in the riskless asset. 
\end{remark}
\noindent Given an $(\bH,\bF)$-admissible strategy $\Psi$, the associated {\em cost process} $C(\Psi)=(C_t(\Psi))_{0 \leq t \leq T}$ is defined by
$$
C_t(\Psi)=V_t(\Psi)-\int_0^t \theta_s \ud S_s, \quad \forall t \in [0,T].
$$
Here $C_t(\Psi)$ describes the total costs incurred by $\Psi$ over the interval $[0,t]$. The $\bH$-{\em risk process} $R^\H(\Psi)=(R_t^\H(\Psi))_{0 \leq t \leq T}$ of $\Psi$ is then defined by
\begin{equation} \label{def:risk}
R_t^\H(\Psi):=\condesph{\left(C_T(\Psi)-C_t(\Psi)\right)^2}, \quad  \forall t \in [0,T].
\end{equation}
Although $(\bH,\bF)$-admissible strategies $\Psi$ with $V_T(\Psi)=\xi$ will in general not be self-financing, it turns out that good $(\bH,\bF)$-admissible strategies are still self-financing on average in the following sense.

\begin{definition}
An $(\bH,\bF)$-admissible strategy $\Psi$ is called {\em mean-self-financing} if its cost process $C(\Psi)$ is an $\bF$-martingale.
\end{definition}
\noindent Inspired by~\cite{s93}, an $(\bH,\bF)$-admissible strategy $\Psi$ is called $(\bH,\bF)$-{\em locally risk-minimizing} if, for any $t < T$, the remaining risk $R^\H(\Psi)$, see \eqref{def:risk}, is minimal under all infinitesimal perturbations of the strategy at time $t$. For further details, we refer to Definition \ref{def:lrm} in Appendix.

\begin{proposition} \label{prop:characterization}
Suppose that $\langle M \rangle$ is $\P$-a.s. strictly increasing. Let $\xi \in L^2(\Omega,\F_T,\P;\R)$ be a contingent claim and $\Psi$ an $(\bH,\bF)$-admissible strategy with $V_T(\Psi)=\xi$ $\P$-a.s.. Then $\Psi$ is $(\bH,\bF)$-locally risk-minimizing if and only if $\Psi$ is mean-self-financing and the $\bF$-martingale $C(\Psi)$ is weakly orthogonal to $M$.
\end{proposition}

\begin{proof}
For the proof, we refer to Section \ref{sec:tech} in Appendix.
\end{proof}

\noindent The previous result motivates the following.
\begin{definition}
Let  $\xi \in L^2(\Omega,\F_T,\P;\R)$ be a contingent claim. An $(\bH,\bF)$-admissible strategy $\Psi$ with $V_T(\psi)=\xi$ $\P$-a.s. is called $(\bH,\bF)$-{\em optimal} for $\xi$ if $\Psi$ is mean-self-financing and the $\bF$-martingale $C(\Psi)$ is weakly orthogonal to $M$. 
\end{definition} 

\noindent The next result ensures that the existence of an $(\bH,\bF)$-{\em optimal} strategy is equivalent to the decomposition \eqref{F-Sweak} of the contingent claim $\xi$. In the case of full information, an analogous result can be found in~\cite{fs}. 

\begin{proposition} \label{Claudia}
 A contingent claim $\xi \in L^2(\Omega,\F_T,\P;\R)$ admits an $(\bH,\bF)$-optimal strategy $\Psi=(\theta,\eta)$ with $V_T(\Psi)=\xi$ $\P$-a.s. if and only if $\xi$ can be written as
\begin{equation} \label{F-Sweak1}
\xi = U_0 + \int_0^T \beta^\H_t  \ud S_t + A_T, \quad \P-\mbox{a.s.},
\end{equation}
with $U_0\in L^2(\Omega,\F_0,\P;\R)$, $\beta^\H \in \mathcal M_\H^2(0,T)$ and $A \in \mathcal L_\F^2(0,T)$
weakly orthogonal to $M$. The strategy $\Psi$ is then given by
$$
\theta_t=\beta_t^\H, \quad 0 \leq t \leq T
$$
with minimal cost
$$
C_t(\Psi)= U_0 +A_t, \quad 0 \leq t \leq T.
$$
If \eqref{F-Sweak1} holds, the optimal portfolio value is
$$
V_t(\Psi)=C_t(\Psi) + \int_0^t\theta_s\ud S_s= U_0 + \int_0^t\beta_s^\H\ud S_s + A_t,  \quad 0 \leq t \leq T
$$
and
$$
\eta_t=V_t(\Psi)-\beta_t^\H S_t, \quad  0 \leq t \leq T.
$$
\end{proposition}

\begin{proof}
Suppose that $\Psi$ is an $(\bH,\bF)$-optimal strategy with $V_T(\Psi)=\xi$ $\P$-a.s.. Then, the replication constraint yields
\begin{equation} \label{eq:FSdecomposition}
\xi=V_T(\Psi)=C_T(\Psi) + \int_0^T\theta_s\ud S_s=C_0(\Psi)+ \int_0^T\theta_s\ud S_s + \left(C_T(\Psi) - C_0(\Psi)\right), \quad \P-{\rm a.s.}.
\end{equation}
Since $\Psi$ is an $(\bH,\bF)$-optimal strategy, by Proposition \ref{prop:characterization} we know that the process $C(\Psi) - C_0(\Psi)$ is a square-integrable $\bF$-martingale weakly orthogonal to $M$ that is in addition null at zero. Hence \eqref{eq:FSdecomposition} is indeed the F\"ollmer-Schweizer decomposition of $\xi$ with respect to $S$ under partial information with $\beta^\H=\theta$ and $A=C(\Psi) - C_0(\Psi)$.\\
We now assume that \eqref{F-Sweak1} holds. Then, we choose 
\begin{align*}
\theta_t & = \beta_t^\H, \quad t \in [0,T],\\
\eta_t & = U_0 + A_t  - \beta_t^\H S_t - \int_0^t \beta_s^\H \ud S_s, \quad t \in [0,T]. 
\end{align*}
Thus, the strategy $\Psi=(\beta^\H, \eta)$ is such that the associated cost is given by
$$
C_t(\Psi) = V_t(\Psi) - \int_0^t \beta_s^\H \ud S_s = U_0 + A_t,
$$
for every $t \in [0,T]$. In particular, $C_T(\Psi)=U_0 + A_T$.
Hence $C(\Psi)$ is an $\bF$-martingale weakly orthogonal to $M$ and this implies that $\Psi$ is an $(\bH,\bF)$-optimal strategy.
\end{proof}

\begin{remark}\label{charct}
As a consequence of Proposition \ref{Claudia} and Theorem \ref{th:ex-uniq}, under Assumptions \ref{ass:bracket} and \ref{ass:alpha}, we can characterize, the $(\bH,\bF)$-optimal strategy $\Psi=(\theta,\eta)$, the optimal portfolio value $V(\Psi)$ and the corresponding minimal cost $C(\Psi)$, in terms of the unique solution $(Y, Z, O)$ to the BSDE \eqref{eq:bsde1}  with the particular choice of $f(t,y,z) = - \alpha_t z$; more precisely, $V(\Psi) = Y$, $\theta= Z$ and $C(\Psi) = O + Y_0$. \end{remark}  

By applying Proposition \ref{nuova} (with the particular choice of $f(t,y,z) = - \alpha_t z$) the 
$(\bH,\bF)$-optimal strategy may be expressed in terms of the solution of a problem under full information. 

\begin{proposition} \label{nuovabis}
Let Assumptions \ref{ass:bracket} and \ref{ass:alpha} hold. Let $(\tilde Y,\tilde Z, \tilde O)\in \mathcal S_\F^2(0,T) \times \mathcal M_\F^2(0,T) \times \mathcal L_\F^2(0,T)$ be a solution to the problem under complete information
\begin{equation} \label{eq:full1}
\tilde Y_t = \xi - \int_t^T  \hat Z_s \alpha_s \ud \langle M\rangle_s - \int_t^T \tilde Z_s \ud M_s - (\tilde O_T - \tilde O_t), \quad 0 \leq t \leq T,
\end{equation}
where $\tilde O$ is strongly orthogonal to $M$ and 
$$ 
\hat Z_t:=  \frac{\ud L^\bH_t}{\ud \langle M\rangle^\bH_t} \quad L_t:=\int_0^t \tilde Z_s \ud \langle M \rangle_s, \quad 0 \leq t \leq T.
$$
Then  the $(\bH,\bF)$-optimal strategy $\Psi=(\beta^\H,\eta)$,  the optimal portfolio value and the minimal cost are given by
$$
 \beta_t^\H = \hat Z_t, 
 \quad  V_t(\Psi) = \tilde Y_t, \quad C_t(\Psi) =  \tilde Y_0 + \tilde O_t +
 \int_0^t ( \tilde Z_s -  \hat Z_s)   \ud M_s \quad  \forall t \in [0,T],
 $$
 respectively.
\end{proposition}

\begin{proof}
By Proposition \ref{nuova} we get the 
the triplet
$$
(Y,Z,O)=\left(\tilde Y, \hat Z, \tilde O + B\right),
$$
where $B= \int( \tilde Z_s -  \hat Z_s)   \ud M_s $ is a square-integrable $\bF$-martingale weakly orthogonal to $M$, 
is a solution to the BSDE \eqref{eq:bsde1} under partial information with the particular choice of $f(t,y,z) = - \alpha_t z$.

Finally, by uniqueness of the solution to this equation   and Remark \ref{charct} the thesis follows.\end{proof}

\begin{remark}
Let us observe that the process $\tilde Z$ coincides with the optimal strategy under full information, $\beta$, only in the particular case where $S$ is an $\bF$-martingale, i.e. $S = M$ (see \cite{ccr}). In fact, in the semimartingale case, $\beta$ is given by the  second component of the solution to the BSDE under full information with the choice $f(t,y,z) = - \alpha_t z$ which differs from equation (\ref{eq:full1}) that is not a BSDE.
\end{remark}

\subsection{Local risk-minimization under complete information} \label{sec:complete}

Under full information  and in the case where the stock price process $S$ has continuous trajectories, the locally risk-minimizing strategy can be computed via the Galtchouk-Kunita-Watanabe decompositon of the contingent claim with respect to the minimal martingale measure (in short MMM) $\P^*$, see e.g. Theorem 3.5 of~\cite{s01}. 
This is a consequence of the fact that the MMM preserves orthogonality, which means that any $(\P,\bF)$-martingale strongly orthogonal to the martingale part of $S$ under $\P$ turns out to be a $(\P^*, \bF)$-martingale strongly orthogonal to $S$ under $\P^*$.  We emphasize that this is no longer true in general if $S$ has jumps. However, we are able to characterize the optimal portfolio value in terms of the MMM for $S$ even in presence of jumps. 

Let us recall the definition of the MMM.

\begin{definition}
An equivalent martingale measure $\P^*$ for $S$ 
with square-integrable density $\ud \P^*/\ud \P$ is called {\em minimal martingale measure} (for $S$) if $\P^*=\P$ on $\F_0$ and if every $(\P, \bF)$-martingale $\widetilde A$ which is square-integrable and strongly orthogonal to the martingale part of $S$ is also a $(\P^*, \bF)$-martingale. We call $\P^*$ {\em orthogonality-preserving} if $\widetilde A$ is also strongly orthogonal to $S$ under $\P^*$.
\end{definition}
From now on we assume an additional condition on the jump sizes of the martingale part $M$ of $S$ which ensures the existence of the MMM for $S$. More precisely, we make the following assumption:  
\begin{equation}\label{salti}
1-\alpha_t \Delta M_t>0 \quad \P-\mbox{a.s.} \quad \forall t \in[0,T].
\end{equation}
Hence by the Ansel-Strickel Theorem, see~\cite{as}, there exists the minimal martingale measure $\P^*$ for $S$ defined by
\begin{equation}\label{mmm}
\frac{\ud \P^*}{\ud \P}\Big{|}_{\mathcal{F}_t}= \tilde L_t:=\mathcal{E}\left(-\int\alpha_r \ud M_r\right)_t, \quad t \in [0,T],
\end{equation}
where $\mathcal{E}$ denotes the Dol\'{e}ans-Dade exponential. 
Let us observe that by Assumptions \ref{ass:bracket} and \ref{ass:alpha}  the following estimate holds
$$\esp{e^{\int_0^T \alpha^2_t \ud \langle M\rangle_t}} = e^{ \bar K \rho(T)}$$
which implies that the nonnegative $(\P, \bF)$-local martingale $\tilde L$ is in fact a square-integrable $(\P, \bF)$-martingale, see e.g.~\cite{ps}.

\begin{proposition}\label{full}
Let Assumptions \ref{ass:bracket}, \ref{ass:alpha} and equation (\ref{salti}) hold, $\xi \in L^2(\Omega,\F_T,\P;\R)$,  $\bH = \bF$  and assume the $\langle M \rangle$ to be $\P$-a.s. strictly increasing. 
Then there exists the (classical) locally risk-minimizing strategy $\Psi=(\theta,\eta)$ for $\xi$ and the optimal portfolio value $V^\F(\Psi)$ can be computed via the MMM as
$$
V^\F_t(\Psi)= E ^{\P^*} [\xi |  \F_t ] \quad \forall t \in[0,T],
$$
where the notation $E^{\P^*}[\cdot|\F_t]$ denotes the conditional expectation with respect to $\F_t$ computed under $\P^*$. 
\end{proposition}

\begin{proof}
First let us observe that $\xi \in L^2(\Omega,\F_T,\P;\R)$ and $\tilde L$ square-integrable $(\P,\bF)$-martingale imply that $\xi \in L^1(\Omega,\F_T,\P^*;\R)$.\\
By Propositions  \ref{th:FS},  \ref{prop:characterization}  and \ref{Claudia} we deduce the existence of the (classical) locally risk-minimizing strategy $\Psi=(\theta,\eta)$. 
Consider the F\"ollmer-Schweizer decomposition of $\xi$ under full information:
\begin{equation} \label{F-S2}
\xi = \widetilde U_0 + \int_0^T \beta_t  \ud S_t + \widetilde A_T, \quad \P-{\rm a.s.},
\end{equation}
where $\widetilde U_0 \in L^2(\Omega,\F_0,\P;\R)$, $\beta$ is an $\bF$-predictable process such that $\esp{\int_0^T \beta_s^2 \ud \langle M\rangle_s} < \infty$ and $\widetilde A \in \mathcal L_\F^2(0,T)$ is strongly orthogonal to $M$.
Then, $\theta=\beta$ in \eqref{F-S2} and the optimal portfolio value $V^\F(\Psi)$ satisfies for each $t \in [0,T]$

\begin{equation} \label{MGv}
V^\F_t(\Psi)= \widetilde U_0 + \int_0^t \beta_u \ud S_u + \widetilde A_t,
\end{equation}

with $\widetilde A \in \mathcal L_\F^2(0,T)$ strongly orthogonal to $M$.
Since $ \int \beta_r \ud M_r$ and $\tilde L$  are $(\P, \bF)$-square integrable martingales, then $ \int\beta_r \ud S_r$ is a $(\P^*,\bF)$-martingale (see the proof of Theorem 3.14 in~\cite{fs}).  Therefore, the definition of MMM yields that the optimal portfolio value $V^\F(\Psi)$ turns out to be a $(\P^*, \bF)$-martingale and as a consequence, we get 
$$ 
V^\F_t(\Psi) = E ^{\P^*} [ V^\F_T(\Psi) |  \F_t ] = E ^{\P^*} [\xi |  \F_t ], \quad t \in [0,T].
$$

\end{proof}

\begin{remark}
Let us observe that such a result cannot be extended to the partial information framework, since in the F\"ollmer-Schweizer decomposition of $\xi$ under partial information (see equation \eqref{F-Sweak})  the $\bF$-martingale $A$ is only weakly orthogonal to $M$ and so $A$ is not in general a $(\P^*, \bF)$-martingale.
\end{remark}

\begin{remark}
Proposition \ref {full} 
may be useful to compute the locally risk-minimizing strategy under full information, since by (\ref{MGv}), it may be expressed using the predictable covariation under $\P$ of $V^\F(\Psi)$ and $S$, i.e.
$$ 
\beta_t= \frac{\ud \langle V^\F(\Psi), S\rangle^\P}{\ud \langle S \rangle^\P},\quad t \in [0,T].
$$
See \cite{ta}  and references therein for explicit solutions in exponential L\' evy models.
\end{remark}

\appendix

\section{Technical Results} \label{sec:tech}

\noindent Here we clarify the concept of an $(\bH,\bF)$-locally risk-minimizing strategy. As the original version given in the case full information, see e.g.~\cite{s93}, this concept translates the idea that changing an optimal strategy over a small time interval should lead to an increase of risk, at least asymptotically. 

\begin{definition}
A {\em small perturbation} is an $(\bH,\bF)$-admissible strategy $\Delta=(\delta,\gamma)$ such that $\delta$ is bounded, the variation of $\int \delta_u \alpha_u \ud \langle M\rangle_u$ is bounded (uniformly in $t$ and $\omega$) and $\delta_T=\gamma_T=0$. For any subinterval $(s,t]$ of $[0,T]$, we then define the small perturbation
$$
\Delta|_{(s,t]}:=\left(\delta \I_{(s,t]},\gamma\I_{[s,t)}\right).
$$
\end{definition}
\noindent To explain the notion of a local variation of an $(\bH,\bF)$-admissible strategy, we consider partitions $\tau=(t_i)_{0 \leq i \leq N}$ of the interval $[0,T]$. Such partitions will always satisfy
$$
0=t_0 < t_1 < \ldots < t_N=T.
$$
\begin{definition}\label{def:lrm}
For an $(\bH,\bF)$-admissible strategy $\Psi$, a small perturbation $\Delta$ and a partition $\tau$ of $[0,T]$, we set
\begin{equation} \label{def:risklim}
r_\H^\tau(\Psi,\Delta):=\sum_{t_i,t_{i+1} \in \tau}\frac{R_{t_i}^\H\left(\Psi + \Delta|_{(t_i,t_{i+1}]}\right) - R_{t_i}^\H(\Psi)}{\condesphti{\langle M \rangle_{t_{i+1}} - \langle M \rangle_{t_{i}}}}\I_{(t_i,t_{i+1}]}.
\end{equation}
The strategy $\Psi$ is called  $(\bH,\bF)$-locally risk-minimizing if
$$
\liminf_{n \to \infty}r_\H^{\tau_n}(\Psi,\Delta) \ge 0, \quad (\P \otimes \langle M\rangle)-{\rm a.e.}\ {\rm on}\ \Omega \times [0,T]
$$
for every small perturbation $\Delta$ and every increasing sequence $(\tau_n)_{n \in \N}$ of partitions of $[0,T]$ tending to identity. 
\end{definition}

\begin{remark}
If an $(\bH,\bF)$-admissible strategy $\Psi=(\theta,\eta)$ is mean-self-financing, that is $C(\Psi)$ is an $\bF$-martingale, $\Psi$ is uniquely determined by $\theta$. Indeed, since by the replication constraint we have
$$
C_T(\Psi)=V_T(\Psi) - \int_0^T \theta_s \ud S_s = \xi - \int_0^T \theta_s \ud S_s,
$$
then, by the mean-self-financing property, we get
$$
C_t(\Psi)=\condespf{\xi - \int_0^T \theta_s \ud S_s}, \quad t \in [0,T].
$$
Hence we can write $C(\theta):=C(\Psi)$ and $R^\H(\theta):=R^\H(\Psi)$. This justifies the notation
$$
r_\H^\tau(\theta,\delta):=\sum_{t_i,t_{i+1} \in \tau}\frac{R_{t_i}^\H\left(\theta + \delta\I_{(t_i,t_{i+1}]}\right) - R_{t_i}^\H(\theta)}{\condesphti{\langle M \rangle_{t_{i+1}} - \langle M \rangle_{t_{i}}}}\I_{(t_i,t_{i+1}]},
$$
where $\tau$ is a partition of $[0,T]$.
\end{remark}

\noindent We now prove the martingale characterization of $(\bH,\bF)$-locally risk-minimizing strategies. \\

\noindent {\em Proof of Proposition \ref{prop:characterization}.} {\bf Step 1.} By using similar arguments to those used in the proof of Lemma 2.2 of~\cite{s93}, first we show that an $(\bH,\bF)$-admissible strategy $\Psi=(\theta,\eta)$ with $V_T(\Psi)=\xi$ $\P$-a.s. is $(\bH,\bF)$-locally risk-minimizing if and only if $\Psi$ is mean-self-financing and 
\begin{equation} \label{eq:riskH}
\liminf_{n \to \infty}r_\H^{\tau_n}(\theta,\delta) \ge 0, \quad (\P \otimes \langle M\rangle)-{\rm a.e.}\ {\rm on}\ \Omega \times [0,T]
\end{equation}
for every bounded $\bH$-predictable process $\delta$ such that 
the variation of $\int \delta_u \alpha_u \ud \langle M\rangle_u$ is bounded with $\delta_T=0$
and every increasing sequence $(\tau_n)_{n \in \N}$ of partitions of $[0,T]$ tending to identity. \\
\noindent Let $\Psi=(\theta,\eta)$ be an $(\bH,\bF)$-admissible mean-self-financing strategy with $V_T(\Psi)=\xi$ $\P$-a.s. such that condition 
\eqref{eq:riskH} is satisfied.
Now, take a small perturbation $\Delta=(\delta,\gamma)$ and a partition $\tau$ of $[0,T]$. For $t_i,t_{i+1} \in \tau$, we get the following relationship between the $(\bH,\bF)$-admissible (but not necessarily mean-self-financing) strategy $\Psi+\Delta|_{(t_i,t_{i+1}]}$ and the 
$(\bH,\bF)$-admissible mean-self-financing strategy associated to $\theta + \delta|_{(t_i,t_{i+1}]}$:
\begin{equation}\label{eq:riskHH}
r_\H^\tau(\Psi,\Delta) = r_\H^{\tau}(\theta,\delta) + \sum_{t_i,t_{i+1} \in \tau} \frac{\left(\gamma_{t_i} + \condesphti{\int_{t_i}^{t_{i+1}}\delta_u\alpha_u\ud \langle M\rangle_u}\right)^2}{\condesphti{\langle M \rangle_{t_{i+1}} - \langle M \rangle_{t_{i}}}}\I_{(t_i,t_{i+1}]}.
\end{equation}
Then, by \eqref{eq:riskH} it immediately follows that $\Psi$ is $(\bH,\bF)$-locally risk-minimizing.\\
For the converse, let $\Psi$ be an $(\bH,\bF)$-locally risk-minimizing strategy. By adapting Lemma 2.1 of~\cite{s93} to our framework, it is not difficult to show that $\Psi$ is also mean-self-financing. It only remains to prove that condition \eqref{eq:riskH} is fulfilled. 
Let us observe that we may choose all $\gamma_{t_i}$ to be $0$ in \eqref{eq:riskHH}.
By Assumptions \ref{ass:bracket} and \ref{ass:alpha}, the following estimates hold:
\begin{align*}
\sum_{t_i,t_{i+1} \in \tau} &\frac{\left(\condesphti{\int_{t_i}^{t_{i+1}}\delta_u\alpha_u\ud \langle M\rangle_u}\right)^2}{\condesphti{\langle M \rangle_{t_{i+1}} - \langle M \rangle_{t_{i}}}}\I_{(t_i,t_{i+1}]}  \\
& \quad \leq \bar{K}^2 \|\delta\|_\infty \sum_{t_i,t_{i+1} \in \tau}\condesphti{\langle M \rangle_{t_{i+1}} - \langle M \rangle_{t_{i}}}\I_{(t_i,t_{i+1}]}\\
& \quad \leq \bar{K}^2 \|\delta\|_\infty \sum_{t_i,t_{i+1} \in \tau} \rho(t_{i+1} - t_i)\I_{(t_i,t_{i+1}]}.
\end{align*}
It is easy to see that this last expression converges to $0$ $(\P \otimes \langle M\rangle)$-a.e. on $\Omega \times [0,T]$. Hence, \eqref{eq:riskH} is satisfied.

\noindent {\bf Step 2.}
We now consider the $\bF$-martingale $C(\theta)=(C_t(\theta))_{0 \leq t \leq T}$ that represents the cost process associated to an $(\bH,\bF)$-admissible mean-self-financing strategy $\Psi=(\theta,\eta)$. 
Since $C(\theta)$ is square-integrable, we can apply Proposition \ref{prop:GKW}
and get the Galthouk-Kunita-Watanabe decomposition of $C_T(\theta)$ with respect to $M$ under partial information, i.e.
\begin{equation} \label{eq:gkw1}
C_T(\theta)=C_0(\theta) + \int_0^T \mu_u^\H \ud M_u + O_T, \quad \P-{\rm a.s.},
\end{equation}
where $\mu^\H \in \mathcal M_\H^2(0,T)$ and $O \in L_\F^2(0,T)$ is weakly orthogonal to $M$. 
For a partition $\tau$ of $[0,T]$, consider the locally perturbed process associated to the $\bF$-martingale $C(\theta)$:
$$
C_t\left(\theta + \delta \I_{(t_i,t_{i+1}]}\right)=\condespf{C_T(\theta) - \int_{t_{i}}^{t_{i+1}} \delta_u \ud S_u}, \quad 0 \leq t \leq T, \quad 0 \leq i \leq N-1.
$$
We need the following auxiliary result.
\begin{lemma} \label{lem:stat}
Suppose that Assumptions \ref{ass:bracket}  and \ref{ass:alpha} are in force. Then the following statements are equivalent:
\begin{enumerate}
\item $
\liminf_{n \to \infty}r_\H^{\tau_n}(\theta,\delta) \ge 0, \quad (\P \otimes \langle M\rangle)-{\rm a.e.}\ {\rm on}\ \Omega \times [0,T]$, for every bounded $\bH$-predictable process $\delta$ such that 
the variation of $\int \delta_u \alpha_u \ud \langle M\rangle_u$ is bounded with $\delta_T=0$
and every increasing sequence $(\tau_n)_{n \in \N}$ of partitions of $[0,T]$ tending to identity.
\item $\mu^\H =0$, $(\P \otimes \langle M\rangle)-{\rm a.e.}\ {\rm on}\ \Omega \times [0,T]$, where $\mu^H$ is given in \eqref{eq:gkw1}. 
\item $C(\theta)$ is weakly orthogonal to $M$.
\end{enumerate}
\end{lemma}
\begin{proof}
First we show that the limit in (1) exists $(\P \otimes \langle M\rangle)$-a.e. on $\Omega \times [0,T]$ and equals $\delta^2 - 2 \delta \mu^\H$. 
Similarly to the proof of Proposition 3.1 of~\cite{s90}, consider the difference
\begin{equation*}
\begin{split}
& C_T\left(\theta + \delta \I_{(t_i,t_{i+1}]}\right) - C_{t_i}\left(\theta + \delta \I_{(t_i,t_{i+1}]}\right)\\
& = C_T(\theta) - C_{t_i}(\theta) - \int_{t_i}^{t_{i+1}} \delta_u \ud M_u - \left(\int_{t_i}^{t_{i+1}} \delta_u \alpha_u \ud \langle M\rangle_u - \condespfti{\int_{t_i}^{t_{i+1}} \delta_u \alpha_u \ud \langle M\rangle_u}\right).
\end{split}
\end{equation*}
Then by \eqref{eq:gkw1} and Lemma 5.4 of~\cite{ccr}, we have
\begin{align*}
&R_{t_i}^\H\left(\theta + \delta\I_{(t_i,t_{i+1}]}\right) - R_{t_i}^\H(\theta)\\
& = \condesphti{\left(C_T\left(\theta + \delta \I_{(t_i,t_{i+1}]}\right) - C_{t_i}\left(\theta + \delta \I_{(t_i,t_{i+1}]}\right)\right)^2} - \condesphti{\left(C_T(\theta) - C_{t_i}(\theta)\right)^2}\\
& = \condesphti{\int_{t_i}^{t_{i+1}}\left(\delta_u^2 - 2\delta_u \mu_u^\H\right)\ud \langle M \rangle_u} + \condesphti{\varfti{\int_{t_i}^{t_{i+1}} \delta_u \alpha_u \ud \langle M\rangle_u}}\\
& \quad \quad + 2 \condesphti{\covfti{\int_{t_i}^{t_{i+1}} \delta_u \ud M_u - (C_{t_{i+1}}(\theta) - C_{t_i}(\theta)), \int_{t_i}^{t_{i+1}} \delta_u \alpha_u \ud \langle M\rangle_u}}.
\end{align*} 
Then, this allows to write the quantity $r_\H^{\tau_n}(\theta,\delta)$ easily as the sum of three terms. By martingale convergence, the term involving the process $\mu^\H$ tends to $\delta^2 - 2\delta \mu^\H$ $(\P \otimes \langle M\rangle)$- a.e. on $\Omega \times [0,T]$, as argued in the proof of Proposition 3.1 of~\cite{s90}. For the second term, we get the following estimate:
\begin{align*}
\sum_{t_i,t_{i+1} \in \tau}&\frac{\condesphti{\varfti{\int_{t_i}^{t_{i+1}} \delta_u \alpha_u \ud \langle M\rangle_u}}}{\condesphti{\langle M \rangle_{t_{i+1}} - \langle M \rangle_{t_{i}}}}\I_{(t_i,t_{i+1}]} \\
& \qquad \leq
\sum_{t_i,t_{i+1} \in \tau} \frac{\condesphti{\condespfti{\left(\int_{t_i}^{t_{i+1}} \delta_u \alpha_u \ud \langle M\rangle_u \right)^2}}}{\condesphti{\langle M \rangle_{t_{i+1}} - \langle M \rangle_{t_{i}}}}\I_{(t_i,t_{i+1}]}\\
& \qquad = \sum_{t_i,t_{i+1} \in \tau} \frac{\condesphti{\left(\int_{t_i}^{t_{i+1}} \delta_u \alpha_u \ud \langle M\rangle_u \right)^2}}{\condesphti{\langle M \rangle_{t_{i+1}} - \langle M \rangle_{t_{i}}}}\I_{(t_i,t_{i+1}]}.
\end{align*}
About the third term, we use the Cauchy-Schwarz inequality for sums and the previous estimate to get
\begin{align*}
& \left|\sum_{t_i,t_{i+1} \in \tau}\frac{\condesphti{\covfti{\int_{t_i}^{t_{i+1}} \delta_u \ud M_u - (C_{t_{i+1}}(\theta) - C_{t_i}(\theta)), \int_{t_i}^{t_{i+1}} \delta_u \alpha_u \ud \langle M\rangle_u}}}{\condesphti{\langle M \rangle_{t_{i+1}} - \langle M \rangle_{t_{i}}}}\I_{(t_i,t_{i+1}]}\right|\\
& \quad = \left|\sum_{t_i,t_{i+1} \in \tau}\frac{\condesphti{\int_{t_i}^{t_{i+1}} \delta_u \ud M_u - \left(C_{t_{i+1}}(\theta) - C_{t_i}(\theta)\right)\int_{t_i}^{t_{i+1}} \delta_u \alpha_u \ud \langle M\rangle_u}}{\condesphti{\langle M \rangle_{t_{i+1}} - \langle M \rangle_{t_{i}}}}\I_{(t_i,t_{i+1}]}\right|\\
& \quad \quad \leq \left(\sum_{t_i,t_{i+1} \in \tau} \frac{\condesphti{\left(\int_{t_i}^{t_{i+1}} \delta_u \alpha_u \ud \langle M\rangle_u \right)^2}}{\condesphti{\langle M \rangle_{t_{i+1}} - \langle M \rangle_{t_{i}}}}\I_{(t_i,t_{i+1}]}\right)^{\frac{1}{2}}\\
& \quad \quad \qquad \cdot \left(\sum_{t_i,t_{i+1} \in \tau}\frac{\condesphti{\int_{t_i}^{t_{i+1}}\delta_u^2\ud \langle M\rangle_u +\left(\langle C(\theta)\rangle_{t_{i+1}}-\langle C(\theta)\rangle_{t_{i}}\right)}}{\condesphti{\langle M \rangle_{t_{i+1}} - \langle M \rangle_{t_{i}}}}\I_{(t_i,t_{i+1}]}\right)^{\frac{1}{2}}.
\end{align*}
By similar arguments to the ones used in the proof of Proposition 3.1 of~\cite{s90}, it is sufficient to show that
$$
\lim_{n \to \infty}\sum_{t_i,t_{i+1} \in \tau} \frac{\left(\int_{t_i}^{t_{i+1}} \delta_u \alpha_u \ud \langle M\rangle_u \right)^2}{\langle M \rangle_{t_{i+1}} - \langle M \rangle_{t_{i}}}\I_{(t_i,t_{i+1}]}=0, \quad (\P \otimes \langle M\rangle)-{\rm a.e.}\ {\rm on}\ \Omega \times [0,T],
$$
due to Lemma 2.1 of~\cite{s90}. By Assumptions \ref{ass:bracket} and \ref{ass:alpha}, we have
\begin{align*}
 \sum_{t_i,t_{i+1} \in \tau} \frac{\left(\int_{t_i}^{t_{i+1}} \delta_u \alpha_u \ud \langle M\rangle_u \right)^2}{\langle M \rangle_{t_{i+1}} - \langle M \rangle_{t_{i}}}\I_{(t_i,t_{i+1}]} & \leq 
\bar{K}^2\|\delta\|_\infty^2 \sum_{t_i,t_{i+1} \in \tau}\left(\langle M \rangle_{t_{i+1}} - \langle M \rangle_{t_{i}}\right)\I_{(t_i,t_{i+1}]}\\
& \leq \bar{K}^2 \|\delta\|_\infty \sum_{t_i,t_{i+1} \in \tau} \rho(t_{i+1} - t_i)\I_{(t_i,t_{i+1}]}
\end{align*}
and the last expression converges to $0$ $(\P \otimes \langle M\rangle)$-a.e. on $\Omega \times [0,T]$.\\
By \eqref{eq:gkw1}, it is easy to check that $C(\theta)$ is weakly orthogonal to $M$ if and only if $\mu^\H =0$ $(\P \otimes \langle M\rangle)$-a.e. on $\Omega \times [0,T]$. It is obvious that (2) implies (1). Since 
$\lim_{n \to \infty}r_\H^{\tau_n}(\theta,\delta)=\delta^2- 2 \delta \mu^\H$ $(\P \otimes \langle M\rangle)-{\rm a.e.}\ {\rm on}\ \Omega \times [0,T]$,
for every bounded $\bH$-predictable process $\delta$ such that 
the variation of $\int \delta_u \alpha_u \ud \langle M\rangle_u$ is bounded with $\delta_T=0$, 
to prove that (1) implies (2), for any $\epsilon > 0$ and $k > 0$ we choose $\delta:=\epsilon \cdot {\rm sign}(\mu^\H) \cdot \I_{\{|\int_0^T \delta_u \alpha_u \ud \langle M \rangle_u| \leq k\}}$. Clearly, $\int \delta_u \alpha_u \ud \langle M\rangle_u$ is bounded and by (1) we deduce that $|\mu^\H|\I_{\{|\int_0^T \delta_u \alpha_u \ud \langle M \rangle_u| \leq k\}} \leq \frac{\epsilon}{2}\I_{\{|\int_0^T \delta_u \alpha_u \ud \langle M \rangle_u| \leq k\}}$, which implies $|\mu^\H|\leq \frac{\epsilon}{2}$ by letting $k \to \infty$.
\end{proof}

\noindent {\bf Step 3.} Finally, by applying Lemma \ref{lem:stat} we obtain the link between condition \eqref{eq:riskH} and the weak orthogonality condition that implies the result. This concludes the proof of Proposition 3.8.
\begin{flushright} 
$\square$
\end{flushright}

 \medskip

\begin{center}
{\bf Acknowledgements}
\end{center}
The third named author was partially supported
by the ANR Project MASTERIE 2010 BLAN 0121 01.

\bibliographystyle{plain}
\bibliography{CCRbibliob}
\end{document}